\documentclass[reqno]{amsart}

\usepackage{amsthm,amsmath,amssymb,amsfonts}
\usepackage[colorlinks=true,hypertexnames=false,linkcolor=blue,citecolor=blue]{hyperref}

\usepackage{epsfig,graphics}
\usepackage{amscd}
\usepackage{amsmath}
\usepackage{bm}
\usepackage{amssymb}
\usepackage{graphicx}
\usepackage{psfrag}
\usepackage{verbatim}
\usepackage{tikz}
\usetikzlibrary{topaths}
\usetikzlibrary{shapes,arrows}

\newtheorem{theorem}{Theorem}[section]

\newtheorem{lemma}[theorem]{Lemma}
\newtheorem{prop}[theorem]{Proposition}
\newtheorem{coro}[theorem]{Corollary}

\theoremstyle{definition}
\newtheorem{definition}[theorem]{Definition}

\newtheorem{question}[theorem]{Question}

\theoremstyle{remark}
\newtheorem{remark}[theorem]{Remark}

\numberwithin{equation}{section}

\theoremstyle{plain}
\newtheorem{maintheorem}{Theorem}

\newcommand{\D}{\ensuremath{\mathbb{D}}}
\newcommand{\Q}{\ensuremath{\mathbb{Q}}}

\newcommand{\R}{\ensuremath{\mathbb{R}}}

\newcommand{\C}{\ensuremath{\mathbb{C}}}
\newcommand{\nt}{\ensuremath{\mathbb{N}}}

\DeclareMathOperator{\Leb}{Leb}

\DeclareMathOperator{\htop}{top}

\DeclareMathOperator{\supp}{supp}

\begin{document}

\title{Real bounds and Lyapunov exponents}

\author{Edson de Faria}
\address{Instituto de Matem\'atica e Estat\'istica, Universidade de S\~ao Paulo}
\curraddr{Rua do Mat\~ao 1010, 05508-090, S\~ao Paulo SP, Brasil}
\email{edson@ime.usp.br}

\author{Pablo Guarino}
\address{Instituto de Matem\'atica e Estat\'istica, Universidade Federal Fluminense}
\curraddr{Rua M\'ario Santos Braga S/N, 24020-140, Niter\'oi, Rio de Janeiro, Brasil}
\email{pablo\_\,guarino@id.uff.br}

\thanks{This work has been supported by ``Projeto Tem\'atico Din\^amica em Baixas Dimens\~oes'' 
FAPESP Grant 2011/16265-2, and by FAPESP Grant 2012/06614-8}



\begin{abstract} We prove that a $C^3$ critical circle map without periodic points has zero Lyapunov exponent 
with respect to its unique invariant Borel probability measure. Moreover, no critical point of such a map 
satisfies the Collet-Eckmann condition. This result is proved directly from the well-known {\it real a-priori bounds\/}, 
without using Pesin's theory. We also show how our methods yield an analogous result for infinitely renormalizable 
unimodal maps of any combinatorial type. Finally we discuss an application of these facts to the study of neutral 
measures of certain rational maps of the Riemann sphere.
\end{abstract}

\maketitle

\vspace{-0.5cm}

\section{Introduction}

This paper studies critical circle maps (as well as infinitely renormalizable unimodal maps) from the differentiable 
ergodic theory viewpoint. 
The ergodic aspects of one-dimensional dynamical systems have been the object of intense research for quite some time. 
In particular, the study of characteristic or Lyapunov exponents of invariant measures, or physical measures, was 
initiated in this context by Ledrappier, Bowen, Ruelle,  and developed by Keller, Blokh and Lyubich among others. 
See \cite[Chapter V]{demelovanstrien} for a full account, and the references therein.

In this article we show that the Lyapunov exponent of a $C^3$ critical circle map (or of an infinitely renormalizable 
unimodal interval map) is always zero. The general approach leading to zero Lyapunov exponents is by arguing by contradiction 
and using 
Pesin's theory: non-zero Lyapunov exponent implies the existence of periodic orbits (see for instance 
\cite[Supplement 4 and 5]{katokhass} 
and \cite[Chapter 11]{livroconf}), and that would be a contradiction for critical circle maps with irrational 
rotation number. 
Our goal in this paper, however, is to prove that the exponent is zero directly from the real a-priori bounds 
(see Theorem \ref{realbounds}), 
without using Pesin's theory. In fact, the only non-trivial result from ergodic theory we shall use here is 
Birkhoff's Ergodic Theorem.

By a \emph{critical circle map} we mean an orientation preserving $C^3$ circle homeomorphism $f$ with finitely many  
non-flat critical points $c_1,c_2,\ldots,c_N$ ($N\geq 1$). A critical point $c$ is called \emph{non-flat} if in a neighbourhood of $c$ the map $f$ 
can be written as $f(t)=\phi(t)\left|\phi(t)\right|^{d-1}+f(c)$, where $\phi$ is a $C^3$ local diffeomorphism 
with $\phi(c)=0$, and where $d>1$ is a real number known as the \emph{criticality} (or \emph{order}, or \emph{type}, 
or \emph{exponent}) of such critical point.
Critical circle maps have been studied by several authors in the last three decades. From a strictly mathematical viewpoint, 
these studies started with basic topological aspects \cite{hall}, 
\cite{yoccoz}, then evolved -- in the special case of maps with a {\it single\/} critical point -- to geometric bounds 
\cite{herman}, 
\cite{swiatek}, and further to geometric rigidity 
and renormalization aspects, see \cite{avila}, \cite{tesisedson}, \cite{edson}, 
\cite{edsonwelington1}, \cite{edsonwelington2}, \cite{tesegua}, \cite{GMdM}, \cite{guamelo},  \cite{khaninteplinsky}, 
\cite{khmelevyampolsky}, \cite{swiatek}, \cite{yampolsky1}, \cite{yampolsky2}, \cite{yampolsky3} and \cite{yampolsky4}. 
The geometric rigidity and renormalization aspects of the theory remain open for maps with more than one critical point, see Question \ref{QS}. Such brief account bypasses important numerical studies by several physicists, as well as computer-assisted 
and conceptual work by Feigenbaum, Kadanoff, Lanford, Rand, Epstein and others; see \cite{edson} and references therein.

As we said before, this paper studies a critical circle map $f$ from the differentiable ergodic theory viewpoint. 
We will focus on the case when the rotation number of $f$ is irrational, in which case $f$ is uniquely ergodic. 
Moreover, by a theorem of Yoccoz \cite{yoccoz}, $f$ is minimal and therefore topologically conjugate to the corresponding 
rigid rotation. This implies that the support of its unique invariant Borel probability measure is the whole circle 
(see Section \ref{integrability} for more details on the invariant measure). Our main result is the following.

\begin{maintheorem}\label{main} Let $f:S^1 \to S^1$ be a $C^3$ critical circle map with irrational rotation 
number, and let $\mu$ be its unique invariant Borel probability measure.
Then $\log Df$ belongs to $L^1(\mu)$ and it has zero mean:
\[
 \int_{S^1}\!\log Df\,d\mu=0.
\]
Moreover, no critical point of $f$ satisfies the Collet-Eckmann condition.
\end{maintheorem}

Recall that $f$ satisfies the \emph{Collet-Eckmann} condition at a critical point $c\in\{c_1,c_2,\ldots,c_N\}$ 
if there exist $C>0$ and $\lambda>1$ such that $Df^n\big(f(c)\big) \geq C\lambda^n$ for all $n\in\nt$ 
(see for instance \cite[Chapter V]{demelovanstrien}), or equivalently
\begin{equation}\label{lyapliminf}
 \liminf_{n\to\infty} \frac{1}{n}\log{Df^n(f(c))}\;\geq\; \log{\lambda}\;>\;0\ .
\end{equation}
The integrability of $\log Df$ was obtained by Przytycki 
in \cite[Theorem B]{p93}, where he also proved that $\int\log Df\,d\mu \geq 0$ (see \cite[Appendix A]{juan} for 
an easier proof). We will obtain the integrability again 
(see Proposition \ref{int}) on the way to proving that $\int\log Df\,d\mu=0$. 
It is expected that $\log Df$ will not be integrable if we allow the presence of flat critical points, 
as in \cite{hall}. 

Theorem \ref{main} applies to some classical examples of holomorphic dynamics in the Riemann sphere, see 
Theorem \ref{exrat} in \S \ref{ratapp}.

\begin{remark} The analogue of Theorem \ref{main} for diffeomorphisms is straightforward: if $f$ is an 
orientation-preserving $C^1$ circle diffeomorphism, with irrational rotation number, the function $\psi:S^1\to\R$ 
defined by $\psi=\log Df$ is a continuous function and therefore, by the unique ergodicity of $f$, the sequence 
of continuous functions:
$$\frac{1}{n}\sum_{j=0}^{n-1}\psi\circ f^j$$ converges uniformly to a constant, and this constant must be $\int_{S^1}\log Df\,d\mu$. 
By the chain rule $\sum_{j=0}^{n-1}\psi\circ f^j=\log Df^n$ and, therefore, the sequence of continuous functions $\log Df^n/n$ 
converges to the constant $\int_{S^1}\log Df\,d\mu$ uniformly in $S^1$. Since $f^n$ is a diffeomorphism for all $n\in\nt$, this 
constant must be zero. In our case, however, $\log Df$ is not a continuous function (it is defined only in $S^1\setminus\{c_1,c_2,\ldots,c_N\}$, 
and it is unbounded, see Figure \ref{logcocycle}).
\end{remark}

\begin{remark} The $C^3$-smoothness hypothesis of Theorem \ref{main} could be relaxed to $C^2$ (or even $C^{1+Zygmund}$, see \cite[Section IV.2.1]{demelovanstrien}) -- indeed whatever smoothness is necessary for the real bounds to hold.
\end{remark}

\begin{figure}[t]
\begin{center}~
\hbox to \hsize{\psfrag{0}[][][1]{$c_i$} \psfrag{S}[][][1]{$S^1$}
\psfrag{1}[][][1]{$\;c_j$}  \psfrag{2}[][][1]{$\;c_k$}
\psfrag{L}[][][1]{$\ \ \ \ \ \ \ \ \ \ \ \psi = \log{Df}$} 
\hspace{1.0em} \includegraphics[width=4.0in]{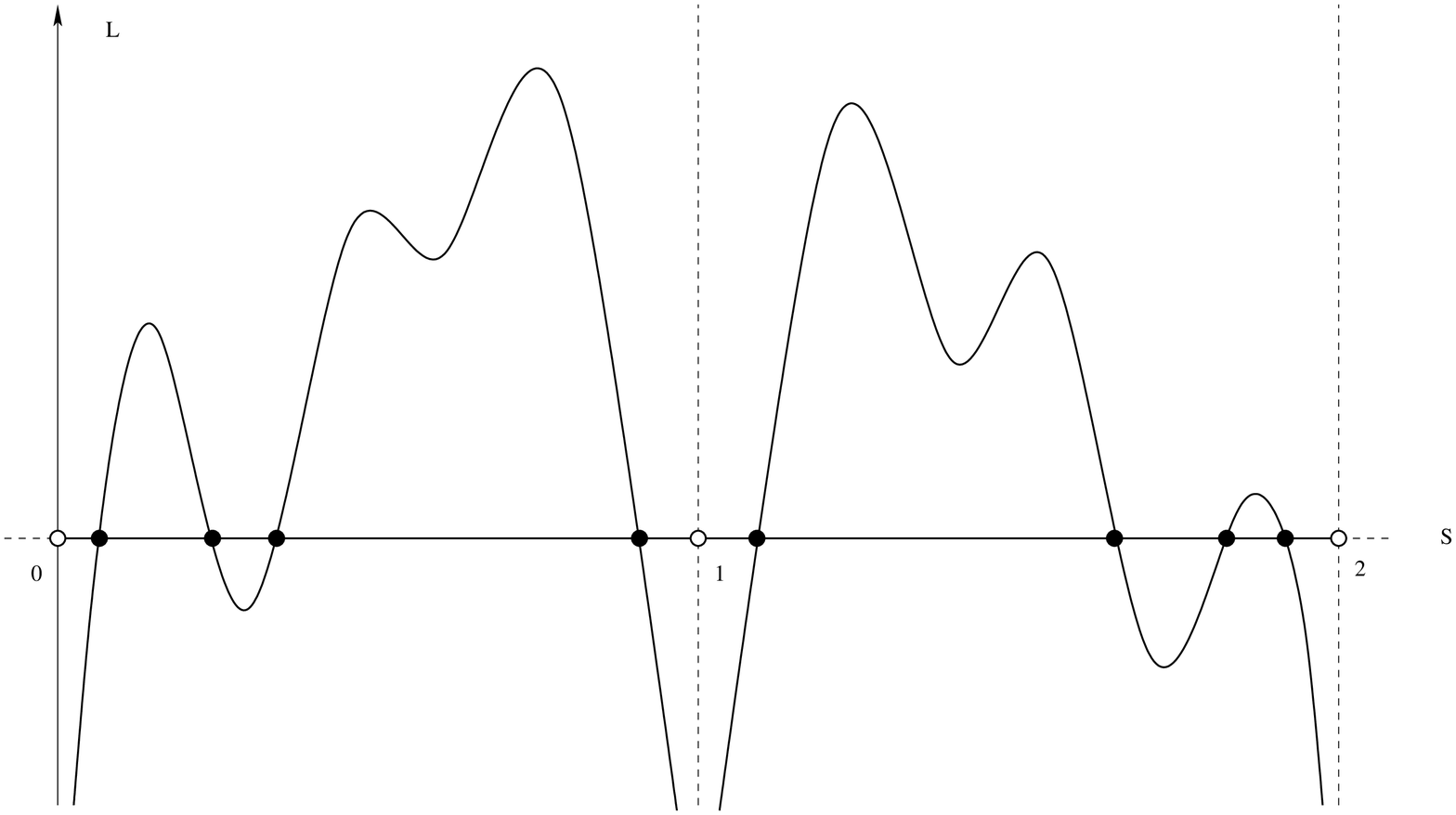}
   }
\end{center}
  \caption[logcocycle]{\label{logcocycle} The cocycle $\psi=\log{Df}$ is unbounded.}
\end{figure}

\subsection{How the paper is organized} In \S \ref{sec:realbounds} we briefly  recall some classical combinatorial facts about circle maps, and also 
the well-known {\it a-priori\/} bounds on the critical orbits of a critical circle map. We deduce from these facts two useful lemmas concerning 
dynamical partitions. In \S \ref{integrability}, we establish the integrability of $\log{Df}$ with respect to the unique $f$-invariant probability 
measure, for any critical circle map without periodic points for which the real bounds of \S \ref{sec:realbounds} hold true. In \S \ref{sec:proofofmain}, 
we use the results of \S \ref{sec:realbounds} and \S \ref{integrability} to prove our main result, namely Theorem \ref{main}. In \S \ref{sec:unimodal}, 
we prove Theorem \ref{unbounded}, an analogous result to Theorem \ref{main} for {\it infinitely renormalizable unimodal maps\/} with non-flat critical 
point. In \S \ref{ratapp}, we discuss an application of Theorem \ref{main} to the ergodic theory of certain Blaschke products as well as quadratic polynomials. Finally, in \S \ref{sec:final}, we conclude by stating a few open questions concerning both critical circle maps and rational maps of the Riemann sphere.

\section{The real bounds}\label{sec:realbounds} Let $f$ be a $C^3$ critical circle map as defined in the introduction, that is, $f$ is an orientation 
preserving $C^3$ circle homeomorphism with finitely many non-flat critical points of odd type. As we have pointed out already, our 
standing assumption is that the rotation number $\rho(f)=\theta\in [0,1)$ is irrational. Therefore it has an infinite continued-fraction expansion, say
$$
\theta\;=\;\left[a_0,a_1,...,a_n,a_{n+1},...\right] \;=\; \lim_{n\to\infty}\,\dfrac{1}{a_0+\dfrac{1}{a_1+\dfrac{1}{a_2+\dfrac{1}{\ddots\dfrac{1}{a_n}}}}}\ .
$$
We define recursively a sequence of \emph{return times} of $\theta$ by:
\begin{center}
$q_0=1,\quad$ $\quad q_1=a_0\quad$ and $\quad q_{n+1}=a_nq_n+q_{n-1}\quad$ for $\quad n \geq 1$.
\end{center}
In particular the sequence $\{q_n\}_{n \geq 1}$ grows at least exponentially fast when $n$ goes to infinity 
(we will use this fact in the proof of Proposition \ref{int} and in the proof of Proposition \ref{Aisgood} below). We recall 
that the numbers $q_n$ are also obtained as the denominators of the truncated expansion of order $n$ of 
$\theta$:
$$
\frac{p_n}{q_n}\;=\;[a_0,a_1,a_2,...,a_{n-1}]\;=\;\dfrac{1}{a_0+\dfrac{1}{a_1+\dfrac{1}{a_2+\dfrac{1}{\ddots\dfrac{1}{a_{n-1}}}}}}\ .
$$
We recall also the following well-known estimates. 

\begin{theorem}\label{convergents} 
For all $n\in\nt$ we have:$$\frac{1}{q_n(q_n+q_{n+1})}<\left|\theta-\frac{p_n}{q_n}\right|\leq\frac{1}{q_nq_{n+1}}<\frac{1}{q_n^2}\ .$$
\end{theorem}

For a proof, see for instance Theorems 9 and 13 in \cite[Ch.~I]{khin}, or Theorem 5 in \cite[Ch.~I]{lang}. 

\subsection{Dynamical partitions}\label{dynpart} Denote by $I_n(c)$ the interval $[c,f^{q_n}(c)]$, where $c \in \{c_1,c_2,\ldots,c_N\}$ denotes a critical 
point of $f$, and define $\mathcal{P}_n(c)$ as:
\[
 \mathcal{P}_n(c)=\big\{f^i(I_n(c)):\,0\leq i\leq q_{n+1}-1\big\} \cup \big\{f^j(I_{n+1}(c)):\,0\leq j\leq q_n-1\big\}\ .
\]

A crucial combinatorial fact is that $\mathcal{P}_n(c)$ is a partition (modulo boundary points) of the circle for every 
$n \in \nt$. We call it the \emph{n-th dynamical partition} of $f$ associated with the point $c$. Note that the 
partition $\mathcal{P}_n(c)$ is determined by the finite piece of orbit$$\big\{f^{j}(c): 0 \leq j \leq q_n + q_{n+1}-1\big\}\,.$$

As we are working with critical circle maps, our partitions in this article are always determined by a critical orbit. 
Our proof of Theorem \ref{main} is based on the following result.

\begin{theorem}[The real bounds]\label{realbounds} There exists a constant $K_0>1$ with the following property.
Given a $C^3$ critical circle map $f$ with irrational rotation number there exists $n_0=n_0(f)$ such that, 
for each critical point $c$ of $f$, for all $n \geq n_0$, 
and for every pair $I,J$ of 
adjacent atoms of $\mathcal{P}_n(c)$ we have:
\begin{equation}\label{comp}
K_0^{-1}|I| \leq |J| \leq K_0|I|,
\end{equation}
where $|I|$ denotes the Euclidean length of an interval $I$ in the real line. 
\end{theorem}

Of course for a particular $f$ we can choose $K_0>1$ such that \eqref{comp} holds for all $n\in\nt$. Theorem \ref{realbounds} was proved by 
\'Swi\c{a}tek and Herman (see \cite{herman} and \cite{swiatek}) in the case when $f$ has a single critical point. The original proof is based on the 
so-called {\it cross-ratio inequality\/} of \'Swi\c{a}tek. As it turns out, this inequality is valid also in the case when the map $f$ has several critical 
points (all of non-flat type), see \cite{petersen}. This fact combined with the method of proof presented in 
\cite[Section 3]{edsonwelington1} yields 
the above general result. A detailed proof will appear in \cite{estevezdefaria}. 

Note that for a rigid rotation we have $|I_n|=a_{n+1}|I_{n+1}|+|I_{n+2}|$. 
If $a_{n+1}$ is big, then $I_n$ is much larger than $I_{n+1}$. Thus, even for rigid rotations, real bounds do not 
hold in general.

In the case of maps with a single critical point, one also has the following corollary, which suitably bounds 
the distortion of first return maps. 

\begin{coro} \label{realbounds2} Given a $C^3$ critical circle map $f$ with irrational rotation number and a unique critical point $c \in S^1$, there exists a constant $K_1>1$ such that the following facts hold true for each $n \geq n_0$:
\begin{enumerate}
\item[(i)] For all $x,y \in f(I_{n+1}(c))$, we have
$$\frac{1}{K_1}\leq\frac{\big|Df^{q_n-1}(x)\big|}{\big|Df^{q_n-1}(y)\big|}\leq K_1\ .$$
\item[(ii)] For all $x,y \in f(I_{n}(c))$, we have
$$\frac{1}{K_1}\leq\frac{\big|Df^{q_{n+1}-1}(x)\big|}{\big|Df^{q_{n+1}-1}(y)\big|}\leq K_1\ .$$
\end{enumerate}
\end{coro}

The control on the distortion of the return maps in the above corollary follows from Koebe's distortion principle (see \cite[Section 3]{edsonwelington1}). When $f$ has two or more critical points, the estimates given in the Corollary may fail, because the intervals $f(I_{n}(c))$ and $f(I_{n+1}(c))$ could in principle contain other critical points of $f^{q_{n+1}-1}$ and $f^{q_n-1}$, respectively. 

\begin{remark}
 We shall henceforth use the constant $K=\max\{K_0,K_1\}>1$ whenever we invoke the real bounds. 
\end{remark}

For our purposes, an important consequence of Corollary \ref{realbounds2} is the following auxiliary result. 

\begin{lemma}\label{In} Let $f$ be as in Corollary \ref{realbounds2}. 
There exists $C>1$ such that for all $n\in\nt$ and for all $x \in I_n(c) \setminus I_{n+2}(c)$:$$\frac{1}{C}\leq Df^{q_{n+1}}(x)\leq C\,.$$
\end{lemma}

\begin{proof}[Proof of Lemma \ref{In}] 
For each $k\in \mathbb{N}$, let us write $I_k$ instead of $I_k(c)$ in this proof. 
Fix $n\in\nt$ and $x \in I_n \setminus I_{n+2}$. By Corollary \ref{realbounds2} the 
map $f^{q_{n+1}-1}:f(I_n) \to f^{q_{n+1}}(I_n)$ has bounded distortion. In particular, there exists $C_0>1$ such that:
$$\frac{1}{C_0}\frac{\big|f^{q_{n+1}}(I_n)\big|}{\big|f(I_n)\big|}\leq Df^{q_{n+1}-1}(f(x)) 
\leq C_0\,\frac{\big|f^{q_{n+1}}(I_n)\big|}{\big|f(I_n)\big|}\,.$$

Since $I_{n+1} \subset f^{q_{n+1}}(I_n) \subset I_n \cup I_{n+1}$ we obtain from the real bounds that 
$(1/K)|I_n| \leq \big|f^{q_{n+1}}(I_n)\big| \leq (1+K)|I_n|$. Therefore:
$$\frac{1}{C_0K}\frac{|I_n|}{\big|f(I_n)\big|}\leq Df^{q_{n+1}-1}(f(x)) \leq C_0(1+K)\frac{|I_n|}{\big|f(I_n)\big|}\,.$$

Since $c$ is a non-flat critical point of $f$ of odd type $d$ there exist $0<C_1<C_2$ such that 
$C_1|I_n|^d\leq|f(I_n)| \leq C_2|I_n|^d$ for all $n\in\nt$, and then:
$$\frac{1}{C_0C_2K}\frac{1}{|I_n|^{d-1}}\leq Df^{q_{n+1}-1}(f(x)) \leq \frac{C_0(1+K)}{C_1}\frac{1}{|I_n|^{d-1}}\,.$$

Again using that $c$ is non-flat there exist $0<A<B$ such that $A|x-c|^{d-1} \leq Df(x) \leq B|x-c|^{d-1}$ for all $x$ 
in a small but fixed neighbourhood around the critical point. In particular, 
$$\big(A/K^{2(d-1)}\big)|I_n|^{d-1} 
\leq Df(x) \leq B|I_n|^{d-1}$$ 
for all $x \in I_n \setminus I_{n+2}$ and for all $n\in\nt$, since $|x-c|\geq|I_{n+2}| 
\geq |I_n|/K^2$, again by the real bounds. With this at hand we deduce that
$$\frac{A}{C_0C_2K^{2d-1}}\leq Df^{q_{n+1}}(x)\leq\frac{BC_0(1+K)}{C_1}\,,$$
for all $x \in I_n \setminus I_{n+2}$ 
and for all $n\in\nt$. This finishes the proof of the lemma, provided we take 
$$C \geq \max\big\{C_0C_2K^{2d-1}/A, BC_0(1+K)/C_1\big\}\,.$$
\end{proof}

The following consequence of the real bounds was inspired by \cite[Lemma A.5, page 379]{edsonwelington1}. It holds under the general assumptions of Theorem \ref{realbounds}, for maps with an arbitrary number of critical points. 
For each $n \geq 1$ let:
$$S_n(c)=\sum_{I\in\mathcal{P}_n(c)\setminus\{I_n(c),I_{n+1}(c)\}}\frac{|I|}{d(c,I)}\,,$$
where $d(c,I)$ denotes the Euclidean distance between an interval $I \subset S^1$ and the critical point $c$.

\begin{lemma}\label{A.5} For each critical point $c$ of $f$, the sequence $\displaystyle{\left\{\frac{S_n(c)}{n}\right\}_{n \geq 1}}$ is bounded.
\end{lemma}

\begin{proof}[Proof of Lemma \ref{A.5}] Given a critical point $c$, let us write in this proof, for simplicity of notation,  $\mathcal{P}_k$, $I_k$ 
instead of $\mathcal{P}_k(c)$, $I_k(c)$ respectively, for each $k\in \mathbb{N}$. 
Note that the transition from $\mathcal{P}_n$ to $\mathcal{P}_{n+1}$ 
can be described in the following way: the interval $I_n=[c,f^{q_n}(c)]$ is subdivided by the points 
$f^{jq_{n+1}+q_n}(c)$ with $1 \leq j \leq a_{n+1}$ into $a_{n+1}+1$ subintervals. 
This sub-partition is spread by the iterates of $f$ to all $f^j(I_n)=f^j([c,f^{q_n}(c)])$ with $0 \leq j < q_{n+1}$. 
The other elements of the partition $\mathcal{P}_n$, namely the intervals $f^j(I_{n+1})$ with $0 \leq j < q_n$, remain unchanged.
On one hand, for any $I\in\mathcal{P}_n\setminus\{I_n,I_{n+1}\}$ we have:
$$\sum_{I \supset J\in\mathcal{P}_{n+1}}\frac{|J|}{d(c,J)}\leq\frac{1}{d(c,I)}\sum_{I \supset J\in\mathcal{P}_{n+1}}|J|=\frac{|I|}{d(c,I)}\,.$$
On the other hand:
$$\sum_{\mathcal{P}_{n+1} \ni J \subset I_n \setminus I_{n+2}}\frac{|J|}{d(c,J)}\leq\frac{1}{|I_{n+2}|}\sum_{\mathcal{P}_{n+1} \ni J \subset 
I_n \setminus I_{n+2}}|J|=\frac{|I_n \setminus I_{n+2}|}{|I_{n+2}|}\,.$$
This gives us:$$0 \leq S_{n+1}-S_n \leq \frac{|I_n \setminus I_{n+2}|}{|I_{n+2}|}\quad\mbox{for all $n \geq 1$.}$$
By the real bounds $\frac{|I_n \setminus I_{n+2}|}{|I_{n+2}|} \leq K^2$ for all $n \geq 1$ and we are done.
\end{proof}

\section{The integrability of $\log{Df}$}\label{integrability}

As before let $f$ be a $C^3$ critical circle map with finitely many non-flat critical points and with rotation number $\rho(f)$. Since we assume that $\rho(f)$ is irrational, $f$ admits a unique invariant Borel probability measure $\mu$. Moreover, by a theorem of Yoccoz \cite{yoccoz}, $f$ has no wandering intervals and therefore there exists a circle homeomorphism $h:S^1 \to S^1$ which is a topological conjugacy between $f$ and the rigid rotation by angle $\rho(f)$, that we denote by $R_{\rho(f)}$. More precisely, the following diagram commutes:
$$
\begin{CD}
(S^1,\mu)@>{f}>>(S^1,\mu)\\
@V{h}VV             @VV{h}V\\
{(S^1,\Leb)}@>>{R_{\rho(f)}}>{(S^1,\Leb)}
\end{CD}
$$
where $\Leb$ denotes the normalized Lebesgue measure in the unit circle (the Haar measure for the multiplicative group of 
complex numbers of modulus $1$). Therefore $\mu$ is just the push-forward of Lebesgue measure under $h^{-1}$, that is, 
$\mu(A)=\big(h^{-1}_{*}\Leb\big)(A)=\Leb\big(h(A)\big)$ for any Borel set $A$ in the unit circle (recall that the conjugacy $h$ 
is unique up to post-composition with rotations, so the measure $\mu$ is well-defined).

In this section we prove that $\log Df$ belongs to $L^1(\mu)$. As before, let us denote by $c_1,c_2,\ldots,c_N$ the critical 
points of $f$. 

Let $\varphi:S^1\to\R$ be given by $\varphi=|\log Df|$. 
For each $1\leq j\leq N$ and each $n \geq 1$, let $J_n(c_j)=I_n(c_j)\cup I_{n+1}(c_j)$. We define $E_n=\bigcup_{j=1}^{N} J_n(c_j)$ and consider 
$\varphi_n:S^1\to\R$ given by:$$\varphi_n=\chi_{S^1\setminus E_n}\cdot\varphi\,,$$that is, 
$\varphi_n=0$ on each $J_n(c_j)$ and $\varphi_n=\varphi$ on the complement of their union. We will use the following four facts:

\begin{enumerate}
\item[{\sl Fact 1.}] From the real bounds (Theorem \ref{realbounds}) there exists $0<\lambda<1$ such that $|I_k(c_j)|\geq \lambda^k$ for all 
$k\geq 1$ and each $1\leq j\leq N$. 

\item[{\sl Fact 2.}] As explained above, the measure $\mu$ is the pullback of the Lebesgue measure under any topological conjugacy between $f$ and the corresponding 
rigid rotation. In particular, for each $1\leq j\leq N$ and for all $k\geq 1$, we have $\mu(I_k(c_j))=|q_k\theta -p_k|$ and by Theorem \ref{convergents}: 
$$\frac{1}{q_k+q_{k+1}}<\mu(I_k(c_j))\leq\frac{1}{q_{k+1}}\quad\mbox{for all $k \geq 1$ and each $1\leq j\leq N$.}$$

\item[{\sl Fact 3.}] By combinatorics,  we have $\mu(I_k(c_j)\setminus I_{k+2}(c_j))= a_{k+1}\mu(I_{k+1}(c_j))$, for all $k\geq 0$ and 
for each $1\leq j\leq N$. 

\item[{\sl Fact 4.}] Since each $c_j$ is a  non-flat critical point, there exist $C_0>0$ and a neighbourhood $V_j$ of $c_j$ such that for all 
$x\in V_j$ we have:
\begin{equation}\label{eq1}
 \varphi(x)\;\leq\; C_0\log{\frac{1}{|x-c_j|}}\ .
\end{equation}
We may assume, of course, that the $V_j$'s are pairwise disjoint. 
\end{enumerate}

With all these facts at hand we are ready to prove the desired integrability result.

\begin{prop}\label{int} The function $\log{Df}$ is $\mu$-integrable, i.e., $\log Df\in L^1(\mu)$. 
\end{prop}

\begin{proof}[Proof of Proposition \ref{int}] Note that the sequence $\{\varphi_n\}$ converges monotonically to $\varphi=|\log Df|$. 
Let $n_0$ be the smallest positive integer such that 
$J_{n_0}(c_j)\subseteq V_j$ for all $1\leq j\leq N$. We only look at values of $n$ greater than $n_0$. 
Then, since $\varphi_n$ is identically zero on $E_n$ and agrees with $\varphi$ everywhere else, we can write
\begin{equation}\label{eq2}
 \int_{S^1} \varphi_n\,d\mu \;=\; \int_{S^1\setminus E_{n_0}} \varphi\,d\mu + 
\sum_{j=1}^{N}\sum_{k=n_0}^{n-1} \int_{I_k(c_j)\setminus I_{k+2}(c_j)} \varphi\,d\mu
\end{equation}
The first integral on the right-hand side is a fixed number independent of $n$. Hence it suffices to bound the last double sum. Using \eqref{eq1} 
and the fact that in $I_k(c_j)\setminus I_{k+2}(c_j)$ the closest point to $c_j$ is $f^{q_{k+2}}(c_j)$, we see that (see Figure \ref{lyaplog})
\begin{equation}\label{eq3}
\sum_{k=n_0}^{n-1} \int_{I_k(c_j)\setminus I_{k+2}(c_j)} \varphi\,d\mu\; 
\leq C_0  \sum_{k=n_0}^{n-1} \mu(I_k(c_j)\setminus I_{k+2}(c_j))\log{\frac{1}{|I_{k+2}(c_j)|}}
\end{equation}
Applying facts 1, 2 and 3 to this last sum, we see that
\begin{equation}\label{eq4}
 \sum_{k=n_0}^{n-1} \mu(I_k(c_j)\setminus I_{k+2}(c_j))\log{\frac{1}{|I_{k+2}(c_j)|}}\;\leq\; C_1\sum_{k=n_0}^{n-1} (k+2)a_{k+1}|q_{k+1}\theta -p_{k+1}|
\end{equation}
However we know from Theorem \ref{convergents} that
\begin{equation}\label{eq5}
 |q_{k+1}\theta -p_{k+1}|\;\leq\;\frac{1}{q_{k+2}}\;=\;\frac{1}{a_{k+1}q_{k+1}+q_k}\;<\;\frac{1}{a_{k+1}q_{k+1}}
\end{equation}
Putting \eqref{eq5} into \eqref{eq4} we get
\begin{equation}\label{eq6}
 \sum_{k=n_0}^{n-1} \mu(I_k(c_j)\setminus I_{k+2}(c_j))\log{\frac{1}{|I_{k+2}(c_j)|}}\;\leq\; C_1\sum_{k=n_0}^{n-1} \frac{(k+2)}{q_{k+1}}\,.
\end{equation}
Since the $q_k$'s grow exponentially fast (at least as fast as the Fibonacci numbers), we have
\[
 \sum_{k=0}^{\infty} \frac{(k+2)}{q_{k+1}} \;<\;\infty \ .
\]
Hence the left-hand side of \eqref{eq6} is uniformly bounded. Taking this information back to \eqref{eq3} and then to \eqref{eq2}, we deduce that
there exists a constant $C_2>0$ such that
\[
 \int_{S^1} \varphi_n\,d\mu \;\leq C_2\quad\mbox{for all $n \geq 1$.}
\]
But then, by the Monotone Convergence Theorem, $\varphi$ is $\mu$-integrable, as desired.
\end{proof}

\begin{figure}[h]
\begin{center}~
\hbox to \hsize{\psfrag{c}[][][1]{$c$} \psfrag{S}[][][1]{$S^1$}
\psfrag{I}[][][1]{$I_{k-1}\!\!\!\setminus\!\! I_{k+1}$} \psfrag{II}[][][1]{$I_k\!\!\setminus\!\! I_{k+2}$}
\psfrag{J}[][][1]{$J_k$} \psfrag{JJ}[][][1]{$J_{k-1}$}
\psfrag{L}[][][1]{$\ \ \ \ \varphi=|\log{Df}|$} 
\hspace{1.0em} \includegraphics[width=4.0in]{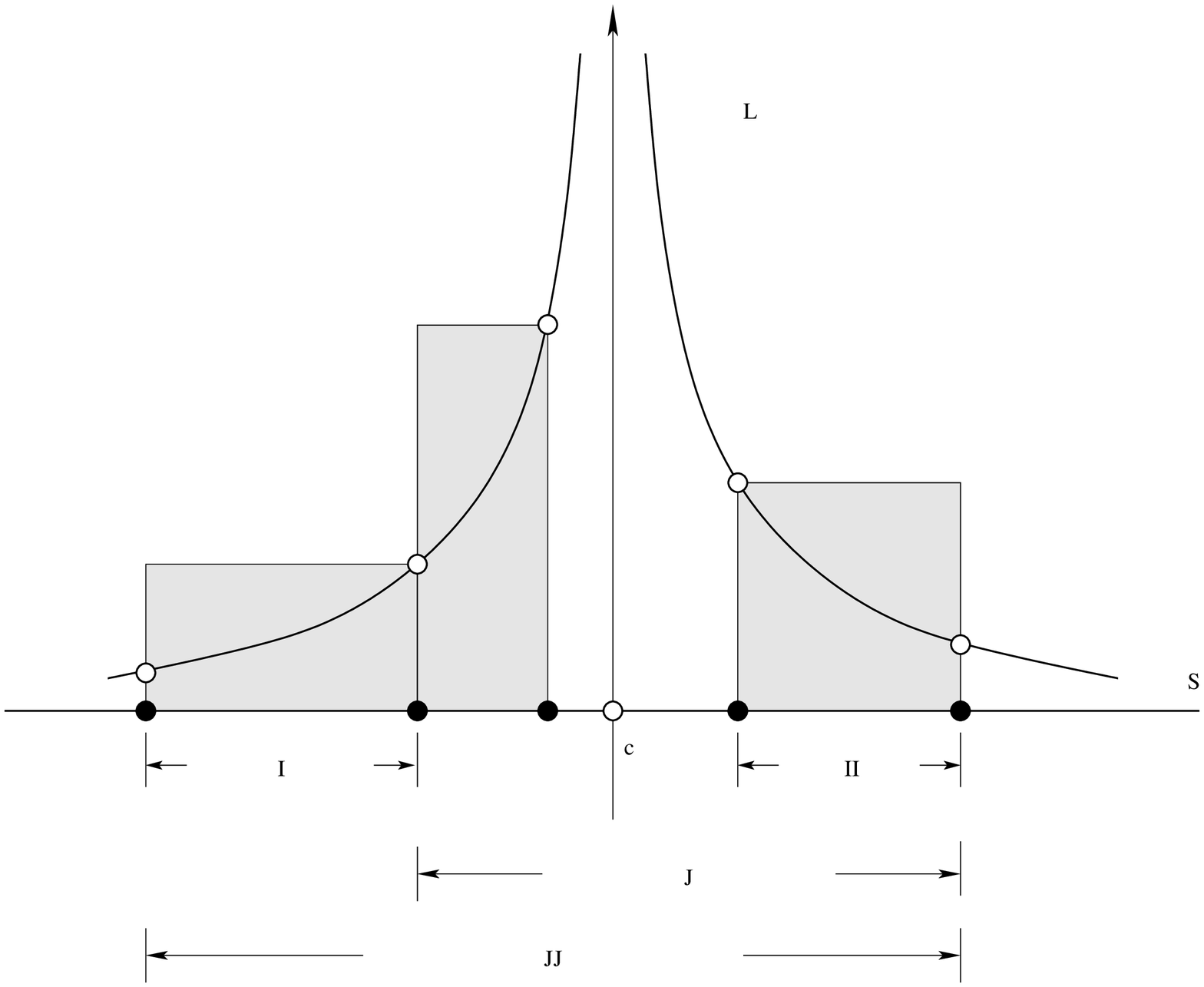}
   }
\end{center}
  \caption[slope]{\label{lyaplog} Bounding the integral of $\varphi =|\log{Df}|$ near a critical point $c$.}
\end{figure}

\begin{remark} The proof of Proposition \ref{int} yields, {\it mutatis mutandis\/}, a slightly stronger result, namely that $\log Df\in L^p(\mu)$ for every finite $p\geq 1$. 
However, this more general fact will not be needed in this paper.
\end{remark} 

\section{Proof of Theorem \ref{main}}\label{sec:proofofmain}

In this section we present two different proofs of Theorem \ref{main}. The first proof works only when the map $f$ has a single critical point, whereas the second works in the general multicritical case. 

\subsection{First proof: the unicritical case} 

Here we assume that $f$ has a single critical point $c$. In particular, we are free to use Lemma \ref{In}.
Once again we write $\mathcal{P}_k, I_k$, instead of $\mathcal{P}_k(c), I_k(c)$, etc., for simplicity of notation.  

Consider the Borel set $A \subset S^1$ defined in the following way: $x \in A$ iff there exists an increasing sequence 
$\{n_k\}_{k\in\nt}\subset\nt$ such that for each $k\in\nt$ there exists (a necessarily unique) $j_k\in\{0,1,...,q_{n_k+1}-1\}$ 
such that $f^{j_k}(x) \in I_{n_k} \setminus I_{n_k+2}$.

\begin{lemma}\label{Afull} The set $A$ is $f$-invariant and $\mu(A)=1$.
\end{lemma}

\begin{proof}[Proof of Lemma \ref{Afull}] The first assertion follows immediately from the definition of $A$, 
hence we focus on proving that 
$A$ has full $\mu$-measure. For each $n \geq 1$ consider the disjoint union:
$$A_n=\bigcup_{i=0}^{q_{n+1}-1}f^{-i}(I_n \setminus I_{n+2})\,.$$
We claim that $\mu(A_n)>1/3$ for all $n \geq 1$. Indeed, $\mu(A_n)=q_{n+1}\,\mu(I_n \setminus I_{n+2})=a_{n+1}q_{n+1}\,\mu(I_{n+1})$, 
since $\mu(I_n \setminus I_{n+2})=a_{n+1}\,\mu(I_{n+1})$. As explained at the beginning of Section \ref{integrability}, the measure $\mu$ is the pullback of the Lebesgue measure under any topological 
conjugacy between $f$ and the corresponding rigid rotation. In particular:
$$\mu(I_{n+1})=q_{n+1}\left|\theta-\frac{p_{n+1}}{q_{n+1}}\right|\,,$$
where $\theta\in\R\setminus\Q$ is the rotation number of $f$. By Theorem \ref{convergents}:
$$\left|\theta-\frac{p_{n+1}}{q_{n+1}}\right|>\frac{1}{q_{n+1}(q_{n+2}+q_{n+1})}$$and then:$$\mu(A_n)>\frac{a_{n+1}q_{n+1}}{q_{n+2}+q_{n+1}}\,.$$
Since $q_{n+2}=a_{n+1}q_{n+1}+q_n<(a_{n+1}+1)q_{n+1}$ we deduce that 
$$\mu(A_n)>\frac{a_{n+1}}{a_{n+1}+2}\,.$$
Since $a_{n+1} \geq 1$ for all $n \geq 0$ we obtain the claim, that is, $\mu(A_n)>1/3$ for all $n \geq 1$. 
Moreover, since:$$A=\limsup_{n\in\nt}A_n=\bigcap_{k \geq 1}\bigcup_{n=k}^{+\infty}A_n\,,$$we have $\mu(A) \geq 1/3$. 
The ergodicity of $\mu$ under $f$ now implies that $\mu(A)=1$, since $A$ is $f$-invariant.
\end{proof}

Now we consider the Borel set $B \subset S^1$ given by $B=A\setminus\mathcal{O}_{f}^{-}(c)$, where $\mathcal{O}_{f}^{-}(c)$ denotes 
the pre-orbit of the critical point, that is:$$\mathcal{O}_{f}^{-}(c)=\big\{c,f^{-1}(c),...,f^{-n}(c),...\big\}\,.$$

\begin{prop}\label{Bfull} The set $B$ has full $\mu$-measure, i.e., $\mu(B)=1$; moreover, the critical value $f(c)$ belongs to $B$.
\end{prop}

\begin{proof}[Proof of Proposition \ref{Bfull}] Since $\mu$ has no atoms, $\mu\big(\mathcal{O}_{f}^{-}(c)\big)=0$  
(recall that $\mu$ is $f$-invariant and $f$ has no periodic orbits). In particular $\mu(B)=1$. The critical point of $f$ belongs 
to $A$ by definition, and by invariance, so does its critical value. Since there are no periodic orbits, 
$f(c)\notin\mathcal{O}_{f}^{-}(c)$ and then $f(c) \in B$.
\end{proof}

The relation between $B$ and the $\mu$-integrability of $\log Df$ is given by the following:

\begin{prop}\label{Aisgood} Let $x \in B$ and let $\{n_k\}_{k\in\nt}$ be its corresponding increasing sequence of natural numbers. 
Then:$$\lim_{k\to+\infty}\left\{\frac{1}{q_{n_k+1}}\sum_{i=0}^{q_{n_k+1}-1}\log Df\big(f^i(x)\big)\right\}=0\,.$$
\end{prop}

\begin{proof}[Proof of Proposition \ref{Aisgood}] Recall that, since $c$ is a non-flat critical point of $f$, 
there exists $L>0$ 
such that for any $x,y \in S^1\setminus\{c\}$ we have:
$$\big|\log Df(x)-\log Df(y)\big|\leq\frac{L|x-y|}{\min\big\{|x-c|,|y-c|\big\}}\,.$$

Let $x \in B$ and let $\{n_k\}_{k\in\nt}$ be its corresponding increasing sequence of natural numbers. 
Recall that for each $k\in\nt$ 
there exists (a necessarily unique) $j_k\in\{0,1,...,q_{n_k+1}-1\}$ such that 
$f^{j_k}(x) \in I_{n_k} \setminus I_{n_k+2}$. Then we have:
\begin{align*}
\frac{\big|\log Df^{q_{n_k+1}}(x)\big|}{q_{n_k+1}}&\leq\frac{\big|\log Df^{q_{n_k+1}}(x)-
\log Df^{q_{n_k+1}}\big(f^{j_k}(x)\big)\big|}{q_{n_k+1}}\notag\\
&+\frac{\big|\log Df^{q_{n_k+1}}\big(f^{j_k}(x)\big)\big|}{q_{n_k+1}}\notag\\
&\leq\frac{1}{q_{n_k+1}}\sum_{i=0}^{j_k-1}\big|\log Df\big(f^i(x)\big)-\log Df\big(f^{i+q_{n_k+1}}(x)\big)\big|+
\frac{\log C}{q_{n_k+1}}\notag\\
&\leq\frac{L}{q_{n_k+1}}\sum_{i=0}^{j_k-1}\frac{\big|f^i(x)-f^{i+q_{n_k+1}}(x)\big|}{\min\big\{|f^i(x)-c|,|f^{i+q_{n_k+1}}(x)-c|\big\}}+\frac{\log C}{q_{n_k+1}}\,,\notag
\end{align*}
where the second inequality is given by Lemma \ref{In}. By combinatorics we have the following facts:

\begin{enumerate}
\item The points $f^i(x)$ and $f^{i+q_{n_k+1}}(x)$ do not belong to $I_{n_k} \cup I_{n_k+1}$ for any $i\in\{0,1,...,j_k-1\}$.
\item For each $i\in\{0,1,...,j_k-1\}$ there exist $\Delta_i$ and $\Delta^*_i$, adjacent elements of the partition $\mathcal{P}_{n_k}$, 
such that the points $f^i(x)$ and $f^{i+q_{n_k+1}}(x)$ belong to $\Delta_i\cup\Delta^*_i$ (the possibility that $\Delta_i=\Delta^*_i$ is not excluded; 
if they are different, we may suppose that $d(c,\Delta_i) \leq d(c,\Delta^*_i)$). By the real bounds: 
$\big|f^i(x)-f^{i+q_{n_k+1}}(x)\big| \leq |\Delta_i|+|\Delta^*_i| \leq (1+K)|\Delta_i|$.
\item If $i \neq l$ in $\{0,1,...,j_k-1\}$ then $\Delta_i \neq \Delta_l$.
\end{enumerate}

Therefore, with the notation of Lemma \ref{A.5}, we have that:
$$\sum_{i=0}^{j_k-1}\frac{\big|f^i(x)-f^{i+q_{n_k+1}}(x)\big|}{\min\big\{|f^i(x)-c|,|f^{i+q_{n_k+1}}(x)-c|\big\}} 
\leq(1+K)\sum_{I\in\mathcal{P}_{n_k}\setminus\{I_{n_k},I_{n_k+1}\}}\frac{|I|}{d(c,I)}\,,$$
and then:
$$\frac{\big|\log Df^{q _{n_k+1}}(x)\big|}{q_{n_k+1}} \leq L(1+K)\left(\frac{S_{n_k}}{q_{n_k+1}}\right)+\frac{\log C}{q_{n_k+1}} 
\quad\mbox{for all $k\in\nt$.}$$
At this point, recall that the sequence $\{q_n\}_{n \geq 1}$ grows at least exponentially fast with $n$, whereas $S_n$ grows at most linearly with $n$, 
by Lemma \ref{A.5}. Hence both terms in the right-hand side of this last inequality go to zero as $k\to\infty$, and we are done. 
\end{proof}

With Proposition \ref{int}, Proposition \ref{Bfull} and Proposition \ref{Aisgood} at hand our main result -- in the {\it unicritical case\/} -- 
follows in a straightforward manner. 

\begin{proof}[Proof of Theorem \ref{main}] By Proposition \ref{int} we already know that $\log Df$ belongs to $L^1(\mu)$. 
Hence by Birkhoff's Ergodic Theorem we have 
$$\lim_{n \to +\infty}\left\{\frac{\log Df^n(x)}{n}\right\}=\int_{S^1}\log Df\,d\mu\,,$$
for $\mu$-almost every $x \in S^1$. Combining this fact with Proposition \ref{Bfull} and Proposition \ref{Aisgood} we obtain:$$\int_{S^1}\log Df\,d\mu=0\,.$$

Finally we have to prove that $f$ does not satisfy the Collet-Eckmann condition. Indeed, if there were constants $C>0$ and $\lambda>1$ 
such that $Df^n\big(f(c)\big) \geq C\lambda^n$ for all $n\in\nt$ we would have:
$$\liminf_{n \to +\infty}\left\{\frac{\log Df^n\big(f(c)\big)}{n}\right\}\geq\log\lambda>0\,,$$
but this is impossible, since by Proposition \ref{Bfull} we know that $f(c)$ belongs to $B$, and this implies by Proposition \ref{Aisgood} 
that there exists a subsequence of $\log Df^n\big(f(c)\big)/n$ converging to zero.
\end{proof}

\begin{question}\label{expcrit} Theorem \ref{main} suggests the question whether $\lim\big\{\log Df^n\big(f(c)\big)/n\big\}=0$. 
For this purpose it would be enough to prove that the limit exists (since $f(c)$ belongs to $B$), for instance by proving that the 
critical value of $f$ is a $\mu$-typical point for the Birkhoff's averages of $\log Df$. Note, however, that this fact does not follow 
directly from the unique ergodicity of $f$ since $\log Df$ is not a continuous function (it is defined only in $S^1\setminus\{c\}$, 
and it is unbounded, see Figure \ref{logcocycle} in the introduction).
\end{question}

\subsection{Second proof: the general multicritical case}\label{sec:newA}

Let us now give a proof of Theorem \ref{main} that works in general. Our proof relies on Proposition \ref{newlemma2.3} below, 
which can be regarded as a suitable replacement for Lemma \ref{In}. 

As before, let $\{q_n\}_{n\in\nt}$ be the sequence of return times given by the irrational 
rotation number of $f$ (see Section \ref{sec:realbounds}). Let us denote by $c_1,c_2,...,c_N$ the critical points of $f$ ($N \geq 1$) 
and let $d_i > 1$ denote the criticality of each $c_i$. Conjugating $f$ by a suitable $C^3$-diffeomorphism (which does not affect its Lyapunov exponent) we may assume that each $c_i$ has an open neighbourhood $V(c_i)$ where $f$ is a \emph{power-law} of the form:
\begin{equation}\label{newApowerlaw}
f(x)=f(c_i)+(x-c_i)|x-c_i|^{d_i-1}\quad\mbox{for all $x \in V(c_i)$.}
\end{equation}
We also assume, of course, that $V(c_i) \cap V(c_j)=\O$ whenever $i \neq j$.

Recall from the real bounds (Theorem \ref{realbounds}) that, for each $c\in\{c_1,c_2,...,c_N\}$, the dynamical partitions 
$\big\{\mathcal{P}_n(c)\big\}_{n\in\nt}$ have the \emph{comparability} property: any two consecutive atoms of $\mathcal{P}_n(c)$ 
have comparable lengths. We will also need the following three further consequences of the real bounds.

\begin{lemma}\label{newA:lema1} There exists $B_0=B_0(f)>1$ such that for each $c\in\{c_1,c_2,...,c_N\}$, for each $n\in\nt$ and 
for each atom $\Delta\in\mathcal{P}_n(c)$ we have:
$$\frac{|\Delta|}{B_0}\leq\big|f^{q_n}(\Delta)\big|\leq B_0|\Delta|\ .$$\qed
\end{lemma}

\begin{lemma}\label{newA:lema2} There exists $B_1=B_1(f)>1$ with the following property: let $\Delta\in\mathcal{P}_n(c)$ and denote 
by $\Delta^*$ the union of $\Delta$ with its two immediate neighbours in $\mathcal{P}_n(c)$. If $0 \leq j < k \leq q_n$ are such 
that the intervals $f^{j}(\Delta^*)$, $f^{j+1}(\Delta^*)$,..., $f^{k-1}(\Delta^*)$ do not contain any critical point of $f$, 
then the map $f^{k-j}:f^j(\Delta) \to f^k(\Delta)$ has distortion bounded by $B_1$, that is:
\begin{equation}\label{newA:BDlema2}
\frac{1}{B_1}\leq\frac{Df^{k-j}(x)}{Df^{k-j}(y)}\leq B_1\quad\mbox{for all $x,y \in f^j(\Delta)$.}
\end{equation}
\end{lemma}

\begin{proof}[Proof of Lemma \ref{newA:lema2}] The real bounds imply that $f^j(\Delta)$ has \emph{space} inside $f^j(\Delta^*)$. 
Moreover, the map $f^{k-j}:f^j(\Delta^*) \to f^k(\Delta^*)$ is a diffeomorphism, and hence \eqref{newA:BDlema2} follows 
from the standard Koebe distortion principle (see for instance \cite[Lemma 2.4, page 348]{edsonwelington1} and the references therein).
\end{proof}

\begin{lemma}\label{newA:lema3} There exists $B_2=B_2(f)>1$ with the following property: if $c \neq c'$ are critical points of $f$ and 
$\Delta\in\mathcal{P}_n(c)$, $\Delta'\in\mathcal{P}_n(c')$ for some $n\in\nt$ are such that $\Delta\cap\Delta'\neq\O$, 
then $B_2^{-1}|\Delta'|\leq|\Delta|\leq B_2|\Delta'|$.
\end{lemma}

\begin{proof}[Proof of Lemma \ref{newA:lema3}] This follows from the combinatorial fact that $\Delta$ is contained in the union of 
two adjacent atoms of $\mathcal{P}_n(c')$, one of which is $\Delta'$, and likewise for $\Delta'$.
\end{proof}

For each $k \geq 0$ and each critical point $c$ we will use the notation 
$J_k(c)=I_k(c)\cup I_{k+1}(c)=\big[f^{q_{k+1}}(c),f^{q_{k}}(c)\big] \ni c$. 
The key to prove Theorem \ref{main} is the following fact.

\begin{prop}\label{newlemma2.3} There exists $C=C(f)>0$ with the following properties:
\begin{enumerate}
\item\label{propa} For each $x \in S^1$ and all $n \geq 0$ we have $\log Df^{q_n}(x) \leq C$.
\item\label{propb} For all $n \geq 0$, if $x \in S^1$ is such that $f^i(x)\notin\bigcup_{j=1}^{j=N}J_{2n}(c_j)$ for all $0 \leq i \leq q_n$, 
then $\log Df^{q_n}(x) \geq -Cn$.
\end{enumerate}
\end{prop}

In what follows we denote by $C_0$, $C_1$, $C_2$, $C_3$,... positive constants (greater than $1$, in fact) depending only on $f$. 
Moreover, for any two positive numbers $a$ and $b$ we use the notation $a \asymp b$ to mean that $C^{-1}a \leq b \leq Ca$ for some 
constant $C>1$ depending only on $f$.

\begin{proof}[Proof of Proposition \ref{newlemma2.3}] Let us fix once and for all a critical point $c\in\{c_1,c_2,...,c_N\}$. 
We assume that $n \geq 0$ is large enough so that each atom of $\mathcal{P}_n(c)$ contains at most one critical point of $f$. 
Let $x \in S^1$ and let $\Delta\in\mathcal{P}_n(c)$ be such that $x\in\Delta$. Let $\Delta^*\supseteq\Delta$ be as in Lemma \ref{newA:lema2}. 
Just by taking $n$ larger still, we may assume that, for $0 \leq k < q_n$, each $f^k(\Delta^*)$ contains at most one critical point of $f$. 
We say that $0 \leq k < q_n$ is a \emph{critical time} for $x$ if $f^k(\Delta^*)$ contains a critical point of $f$. 
Let us write $0 \leq k_1<k_2<...<k_m < q_n$ for the sequence of all critical times for $x$. Note that $m \leq 3N$ 
since the family $\big\{f^k(\Delta^*)\big\}_{0 \leq k < q_n}$ has intersection multiplicity equal to $3$. 
Using these critical times and the chain rule we can write:
\begin{align}\label{cadeia}
Df^{q_n}(x)=
Df^{k_1}(x)\!&\left[\prod_{j=1}^{m-1}Df^{k_{j+1}-k_j-1}\big(f^{k_j+1}(x)\big)Df\big(f^{k_j}(x)\big)\right] \\
 & \times Df^{q_n-k_m-1}\big(f^{k_m+1}(x)\big)Df\big(f^{k_m}(x)\big).\notag
\end{align}

We proceed to estimate each term in the product \eqref{cadeia} above. From Lemma \ref{newA:lema2} (with $j=0$ and $k=k_1$) we have:
\begin{equation}\label{newAprop1}
Df^{k_1}(x)\asymp\frac{\big|f^{k_1}(\Delta)\big|}{|\Delta|}\,.
\end{equation}

Again from Lemma \ref{newA:lema2} (with $k_j+1$ and $k_{j+1}$ replacing $j$ and $k$ respectively) we have for all $j\in\{1,...,m-1\}$:
\begin{equation}\label{newAprop2}
Df^{k_{j+1}-k_j-1}\big(f^{k_j+1}(x)\big)\asymp\frac{\big|f^{k_{j+1}}(\Delta)\big|}{\big|f^{k_{j}+1}(\Delta)\big|}\,.
\end{equation}

For each $j\in\{1,...,m\}$ let $\beta_j\in\{c_1,c_2,...,c_N\}$ be the (unique) critical point of $f$ in $f^{k_j}(\Delta^*)$, and let $d_j$ be its criticality. Since we are assuming that $n$ is sufficiently large, we may suppose that $f^{k_j}(\Delta^*) \subseteq V(\beta_j)$ for all $j\in\{1,...,m\}$. Then, from the power-law expression \eqref{newApowerlaw} we have:
\begin{equation}\label{newAprop3}
Df\big(f^{k_j}(x)\big)\asymp\big|f^{k_j}(x)-\beta_j\big|^{d_j-1}\!,
\end{equation}
and recall that $d_j-1>1$ for all $j\in\{1,...,m\}$. Still using the power-law expression we see that:
\begin{equation}\label{newAprop4}
\big|f^{k_j+1}(\Delta)\big|\asymp\big|f^{k_j}(\Delta)\big|^{d_j}\quad\mbox{for all $j\in\{1,...,m\}$.}
\end{equation}

Using Lemma \ref{newA:lema2} yet again, we also see that:
\begin{equation}\label{newAprop5}
Df^{q_n-k_m-1}\big(f^{k_m+1}(x)\big)\asymp\frac{\big|f^{q_n}(\Delta)\big|}{\big|f^{k_m+1}(\Delta)\big|}\,.
\end{equation}

Let us now prove item \eqref{propa} of the conclusion of Proposition \ref{newlemma2.3}. Note that \eqref{newAprop3} yields:
\begin{equation}\label{newAprop8}
Df\big(f^{k_j}(x)\big) \leq C_0\big|f^{k_j}(\Delta)\big|^{d_j-1}\quad\mbox{for all $j\in\{1,...,m\}$,}
\end{equation}
where $C_0=C_0(f)>0$. Combining all these facts, namely \eqref{newAprop1}-\eqref{newAprop8}, we deduce the following (upper) telescoping estimate:
\begin{align}\label{tel}
Df^{q_n}(x)&\leq C_1\frac{\big|f^{k_1}(\Delta)\big|}{|\Delta|}\left[\prod_{j=1}^{m-1}\frac{\big|f^{k_{j+1}}(\Delta)\big|}{\big|f^{k_{j}+1}(\Delta)\big|}\big|f^{k_j}(\Delta)\big|^{d_j-1}\right]\big|f^{k_m}(\Delta)\big|^{d_m-1}\frac{\big|f^{q_n}(\Delta)\big|}{\big|f^{k_m+1}(\Delta)\big|}\notag\\
&\asymp\frac{\big|f^{k_1}(\Delta)\big|}{|\Delta|}\left[\prod_{j=1}^{m-1}\frac{\big|f^{k_{j+1}}(\Delta)\big|}{\big|f^{k_j}(\Delta)\big|}\right]\frac{\big|f^{q_n}(\Delta)\big|}{\big|f^{k_m}(\Delta)\big|}=\frac{\big|f^{q_n}(\Delta)\big|}{|\Delta|}\leq C_2\,,
\end{align}
where in the last line we have used \eqref{newAprop4} and finally Lemma \ref{newA:lema1}. This proves item \eqref{propa}. In order to prove item \eqref{propb} note first that all estimates provided above are two-sided, except \eqref{newAprop8}. In order to get a lower bound for the left side of \eqref{newAprop8} we use the hypothesis in \eqref{propb}. Since $f^{k_j}(x) \notin J_{2n}(\beta_j)$ we have:
\begin{equation}\label{newAprop10}
\big|f^{k_j}(x)-\beta_j\big| \geq C_3\big|I_{2n}(\beta_j)\big|.
\end{equation}

From the real bounds we know that there exists $\lambda\in(0,1)$ depending only on $f$ such that $C_4^{-1}\lambda^n\big|I_{n}(\beta_j)\big|\leq\big|I_{2n}(\beta_j)\big|\leq C_4\lambda^n\big|I_{n}(\beta_j)\big|$. Moreover, we claim that $\big|I_{n}(\beta_j)\big|$ is comparable to $\big|f^{k_j}(\Delta)\big|$. Indeed, this follows from Lemma \ref{newA:lema3} because $I_n(\beta_j)\in\mathcal{P}_n(\beta_j)$ intersects an atom of $\mathcal{P}_n(c)$ in $f^{k_j}(\Delta^*)$, and this atom has length comparable to $\big|f^{k_j}(\Delta)\big|$ (such atom is either $f^{k_j}(\Delta)$ itself, or one of its neighbours). Using these facts in \eqref{newAprop10} we deduce that:
\begin{equation}\label{newAprop11}
Df\big(f^{k_j}(x)\big) \geq C_5\,\lambda^{n(d_j-1)}\big|f^{k_j}(\Delta)\big|^{d_j-1}\!.
\end{equation}

Using this lower estimate in place of the upper estimate \eqref{newAprop8} and proceeding as in \eqref{tel} we arrive at the estimate
\begin{equation}
Df^{q_n}(x) \geq C_6\,\lambda^{n(d_1+d_2+...+d_m-m)}\,,
\end{equation}
where again $C_6=C_6(f)>1$. Note that $0<d_1+d_2+...+d_m-m<3(d_1+d_2+...+d_N)$, and since $\alpha=3(d_1+d_2+...+d_N)$ is a positive constant 
depending only on $f$ we get:$$Df^{q_n}(x) \geq C_6\,\lambda^{n\alpha}\,,$$and then:$$\log Df^{q_n}(x) \geq -n\alpha\log\frac{1}{\lambda}+\log C_6 \geq -C_7\,n\,.$$
\end{proof}

With Proposition \ref{newlemma2.3} at hand we are ready to prove Theorem \ref{main} in the general multicritical case.

\begin{proof}[Proof of Theorem \ref{main}] The fact that no critical point of $f$ satisfies the Collet-Eckmann condition follows at once from item \eqref{propa} 
of Proposition \ref{newlemma2.3}. By Proposition \ref{int} we know that $\log Df \in L^1(\mu)$, and then we know from Birkhoff's ergodic 
theorem that:$$\lim_{n \to +\infty}\left\{\frac{\log Df^n(x)}{n}\right\}=\int_{S^1}\log Df\,d\mu\,,$$
for $\mu$-almost every $x \in S^1$. For each $n \geq 0$ let:
\begin{align*}
A_n&=S^1\!\setminus\!\bigcup_{j=1}^{j=N}\bigcup_{i=0}^{q_n-1}f^{-i}\big(J_{2n}(c_j)\big)\\
&=\left\{x \in S^1:\forall\,\,\, 0 \leq i \leq q_n-1:f^i(x) \in S^1\!\setminus\!\bigcup_{j=1}^{j=N}J_{2n}(c_j)\right\},
\end{align*}
and consider$$A=\limsup_{n\in\nt}A_n=\bigcap_{k=1}^{+\infty}\bigcup_{n=k}^{+\infty}A_n\,.$$

We claim that $A$ has full $\mu$-measure. Indeed, since:$$\mu\big(J_{2n}(c_j)\big)=\mu\big(I_{2n}(c_j)\big)+\mu\big(I_{2n+1}(c_j)\big)\leq\frac{1}{q_{2n+1}}+\frac{1}{q_{2n+2}}\,,$$we deduce that $q_n\mu\big(J_{2n}(c_j)\big) \to 0$ (exponentially fast in $n$, in fact) and since $\mu(A_n) \geq 1-Nq_n\mu\big(J_{2n}(c_j)\big)$ we see that $\mu(A_n) \to 1$ as $n\to+\infty$. This implies the claim that $\mu(A)=1$. Now for each $x \in A$ we have from Proposition \ref{newlemma2.3} that there exists a sequence $n_k\to+\infty$ such that:$$\frac{-Cn_k}{q_{n_k}}\leq\frac{\log Df^{q_{n_k}}(x)}{q_{n_k}}\leq\frac{C}{q_{n_k}}\,,$$and letting $k\to+\infty$ we get that:$$\lim_{k\to+\infty}\frac{\log Df^{q_{n_k}}(x)}{q_{n_k}}=0\,.$$

Therefore:$$\lim_{n \to +\infty}\left\{\frac{\log Df^n(x)}{n}\right\}=0$$for $\mu$-almost every $x \in A$, 
and then we are done since $A$ has full $\mu$-measure.
\end{proof}

\begin{remark}
 Note that for maps with two or more critical points the analogue of Question \ref{expcrit} has, in general, a negative answer. 
Indeed, it may well happen that one of the critical points, say $c_1$, lies in the pre-orbit of one of the other critical points. In that case, 
the Lyapunov exponent of $f(c_1)$ will be equal to $-\infty$. 
\end{remark}

\section{Analogous results for unimodal maps}\label{sec:unimodal}

The main result of this paper, Theorem \ref{main}, has an analogue in the context of {\it infinitely renormalizable unimodal 
maps\/}. We wish to state the result (see Theorem \ref{unbounded} below) and this will involve a slight digression into the 
renormalization theory of unimodal maps. We refer the reader to \cite[Chapter VI]{demelovanstrien} for general background on 
the subject. 

Let $I_0=[-1,1]\subset \mathbb{R}$, and consider $C^3$ maps $f:I_0\to I_0$ which are {\it unimodal\/} with $f'(0)=0$, 
$f(0)=1$, {\it i.e.\/}, with a unique critical point at $0$ and critical value at $1$. Without loss of generality 
assume that $f$ is \emph{even}, in the sense that $f(-x)=f(x)$ for all $x \in I_0$. We assume throughout that the 
critical point is {\it non-flat\/}, as defined in the introduction. Such an $f$ is said to be {\it renormalizable\/} 
if there exist $p=p(f)>1$ and $\lambda=\lambda(f)=f^p(0)$ such that $f^p|[-|\lambda|,|\lambda|]$ is unimodal and 
maps $[-| \lambda|,|\lambda|]$ into itself. With $p$ smallest possible, the \emph{first renormalization\/} of $f$ 
is the map $Rf:I_0\to I_0$ given by
\begin{equation}
Rf(x)\;=\;\frac{1}{\lambda}\,f^p(\lambda x)
\end{equation}
The intervals $\Delta_j=f^j([-|\lambda|,|\lambda|])$, for $0\leq j\leq p-1$, have pairwise disjoint interiors, 
and their relative order inside $I_0$ determines a {\it unimodal} permutation $\theta$ of $\{0,1,\ldots,p-1\}$. 
Thus, renormalization consists of a first return map to a small 
neighbourhood of the critical point rescaled to unit size via a linear rescale (see Figure \ref{renorm}). 

\begin{figure}[ht]
\begin{center}{
\begin{tikzpicture}[out=45,in=135,relative]
\draw[line width=0.5pt] (0,0) to[line to] node {x} (9,0);
\draw[line width=2pt] (4,0) to[line to] (5,0);
\draw[line width=2pt] (2,0) to[line to] (3,0);
\draw[line width=2pt] (6,0) to[line to] (7,0);
\draw[line width=2pt] (8,0) to[line to] (9,0);
\draw[->] (4.5,0.3) to[curve to] node[above] {$f$} (8.5,0.3);
\draw[->] (8.5,-0.3) to[curve to] node[below] {$f$} (2.5,-0.3);
\draw[->] (2.5,0.3) to[curve to] node[above] {$f$} (6.5,0.3);
\draw[->] (6.5,-0.3) to[curve to] node[below] {$f$} (4.5,-0.3);
\draw[->,dashed] (4,0) to[line to] node[left] {} (0,-6);
\draw[->,dashed]  (5,0) to[line to] node[right] {$\Lambda_f^{-1}$} (9,-6);
\draw[line width=2pt] (0,-6) to[line to] (9,-6);
\draw[->,out=100,in=80] (3.2,-5.7) to[curve to]  node[above] {$\mathcal{R}f=\Lambda_f^{-1}\circ f^p\circ \Lambda_f$} (5.8,-5.7);
\end{tikzpicture}
}
\end{center}
\caption[renorm]{\label{renorm} Renormalizing a unimodal map.}
\end{figure}

Since $Rf$ is again a normalized unimodal map, one can ask whether $Rf$ is also renormalizable, and if 
the answer is yes then one can define $R^2f=R(Rf)$, and so on.  Thus, a unimodal map $f$ is said to be 
{\it infinitely renormalizable\/} if the entire sequence $f, Rf, R^2f, \ldots, R^nf, \ldots$ is well-defined. 

We assume henceforth that $f$ is infinitely renormalizable. Let us denote by $\mathcal{I}_f\subseteq I_0$ the closure 
of the orbit of the critical point of $f$. The set $\mathcal{I}_f$ is a Cantor set with zero Lebesgue measure, and is 
the {\it global attractor\/} of $f$ both from the \emph{topological} and \emph{metric} points of view: the set 
$\mathcal{B}(\mathcal{I}_f)=\big\{x \in I_0:\omega(x)=\mathcal{I}_f\big\}$ is a residual set in $I_0$, and it has 
full Lebesgue measure (see \cite[Section V.1]{demelovanstrien} and the references therein). 
Let us point out, however, that $\mathcal{B}(\mathcal{I}_f)$ has empty interior.

For each $n\ge 0$, we can write$$R^n f (x) = \frac{1}{\lambda_n} \cdot f^{q_n} (\lambda_n x)$$where 
$q_0=1$, $\lambda_0=1$, $q_n=\prod_{i=0}^{n-1} p(R^if)$ and $\lambda_n=\prod_{i=0}^{n-1}  \lambda (R^if)=f^{q_n}(0)$. 
The positive integers $a_i=p(R^if)\geq 2$ are called the {\it renormalization periods\/} of $f$, and the 
$q_n$'s are the {\it closest return times\/} of the orbit of the critical point (in perfect analogy with the case of 
critical circle maps). Note that $q_{n+1}=a_nq_n=\prod_{i=0}^{i=n}a_i \geq 2^{n+1}$, in particular the sequence 
$q_n$ goes to infinity at least exponentially fast.

Next, consider the {\it renormalization intervals\/} 
$\Delta_{0,n}=[-|\lambda_n|,|\lambda_n|]
\subset I_0$, and
$\Delta_{i,n}=f^i(\Delta_{0,n})$ for
$i=0,1, \ldots, q_n-1$.
The collection
${\mathcal C}_n=\{\Delta_{0,n}, \ldots, \Delta_{q_n-1,n}\}$
consists of pairwise disjoint intervals.
Moreover,
$\bigcup \{\Delta:\Delta\in {\mathcal C}_{n+1}\} \subseteq \bigcup
\{\Delta:\Delta\in {\mathcal C}_n\}$ for all $n \geq 0$ and we have
$$
\mathcal{I}_f=
\bigcap _{n=0}^{\infty}
\bigcup_{i=0}^{q_n-1}
\Delta_{i,n} \ .
$$

It is also well-known that $f|_{\mathcal{I}_f}$ is a homeomorphism, and that the dynamics of $f$ restricted 
to $\mathcal{I}_f$ is conjugate to that of an {\it adding machine\/} (see \cite[Section III.4, Prop. 4.5]{demelovanstrien}). 
More precisely, consider the following inverse system of cyclic groups 
(each endowed with the discrete topology): 
\[
\cdots \to \mathbb{Z}_{a_0a_1\cdots a_na_{n+1}}\to \mathbb{Z}_{a_0a_1\cdots a_n}\to \cdots \to \mathbb{Z}_{a_0a_1} 
\to \mathbb{Z}_{a_0} \ . 
\]
Here, the morphisms obviously correspond to multiplication by the successive periods $a_n$. 
The inverse limit $\mathbb{A}$ of this system together with the translation $\tau$ induced by $x\mapsto x+1$ 
on each factor (with carryover to the left) is a compact abelian group, called the 
{\it adding machine with periods $(a_0,a_1,\ldots)$\/} ({\it cf.\/} \cite[p.~212]{livroconf}). 
The system $(\mathbb{A},\tau)$ is a minimal and uniquely ergodic dynamical system (and so is $f|_{\mathcal{I}_f}$). 
The assertion is that there exists a homeomorphism $H:\mathbb{A}\to \mathcal{I}_f$ such that the diagram
$$
\begin{CD}
\mathbb{A}@>{\tau}>>\mathbb{A}\\
@V{H}VV             @VV{H}V\\
{\mathcal{I}_f}@>>{f}>{\mathcal{I}_f}
\end{CD}
$$
commutes. In particular, if $\nu$ denotes the unique invariant probability measure under $\tau$, then the 
pushforward $\mu=H_*\nu$ is the unique invariant probability measure for $f|_{\mathcal{I}_f}$. It 
is not difficult to check that this measure satisfies
\[
\mu(\Delta_{i,n}\cap \mathcal{I}_f)=\frac{1}{q_n}=\frac{1}{a_0a_1\cdots a_{n-1}}
\]
for each $n\geq 0$ and each $0\leq i\leq q_n-1$. This occurs simply because, at each level $n$, the intervals 
$\Delta_{i,n}$ with $0\leq i\leq q_n-1$ are permuted by $f$. 

Since $\mathcal{B}(\mathcal{I}_f)$ has full Lebesgue measure in $I_0$ and $f|_{\mathcal{I}_f}$ is uniquely ergodic, 
it is easy to check that the measure $\mu$ is the unique \emph{physical} measure for $f$ in $I_0$, 
see \cite[Section V.1, Theorem 1.6]{demelovanstrien}. Now, the exact analogue of Theorem \ref{main} can be stated as follows.

\begin{maintheorem}\label{unbounded} Let $f:I_0 \to I_0$ be a $C^3$ infinitely renormalizable unimodal map with 
non-flat critical point, and let $\mu$ be the unique $f$-invariant Borel probability measure supported in the 
closure of its post-critical set. Then $\log |Df|$ belongs to $L^1(\mu)$ and it has zero mean:
\[
 \int_{I_0}\log |Df|\,d\mu=0.
\]
Moreover, $f$ does not satisfy the Collet-Eckmann condition.
\end{maintheorem}

Note that Theorem \ref{unbounded} is stated for $C^3$ infinitely renormalizable unimodal maps of any combinatorial type, 
with non-flat critical point of any criticality and without any assumption about the Schwarzian derivative.

If $f$ has negative Schwarzian derivative and non-degenerate critical point (that is, $D^2\!f(c) \neq 0$), 
Theorem \ref{unbounded} goes back to Keller \cite[Theorem 3, page 722]{keller}. It is also well-known that an infinitely 
renormalizable unimodal map does not satisfy the Collet-Eckmann condition (otherwise it would admit an absolutely 
continuous invariant probability measure, see \cite{acip} and the references therein, which is impossible since 
its non-wandering set has zero Lebesgue measure \cite[Section VI.2, Theorem 2.1]{demelovanstrien}, see also \cite{lishen}). 
The lack of the Collet-Eckmann condition will be re-obtained here, just as in Theorem \ref{main}, by showing that there 
exists a subsequence of $\log\big|Df^n\big(f(c)\big)\big|/n$ converging to zero.

\begin{remark}
Just as in the case of critical circle maps, Theorem \ref{unbounded} suggests the question whether 
$\lim\big\{\log\big|Df^n\big(f(c)\big)\big|/n\big\}=0$, see Question \ref{expcrit}. Again, for this purpose 
it would be enough to prove that the limit exists, for instance by proving that the critical value of $f$ is a 
$\mu$-typical point for the Birkhoff's averages of $\log|Df|$. 
We remark that Nowicki and Sands have proven in \cite{nowsands} 
that 
\[
\liminf_{n\to\infty}\left\{\frac{1}{n}\log{\left|Df^n\left(f(c)\right)\right|}\right\}\;\geq\; 0\ .
\] 
Combined with what we said above and will prove below, this fact implies that, if the limit exists, it must 
be equal to zero. 
\end{remark}

\subsection{Proof of Theorem \ref{unbounded}} Sullivan has shown in \cite{sullivan} that, at every level of renormalization, 
each renormalization interval $\Delta_{j,n}$ has a definite space around itself inside $I_0$. We can state this particular result 
by Sullivan as follows ({\it cf.\/} \cite[Section VI.2, Lemma 2.1]{demelovanstrien}). Given a closed interval $\Delta$ and $\delta>0$, we 
define the {\it $\delta$-scaled neighborhood\/} of $\Delta$ to be the open interval $V\supset \Delta$ such that each component of 
$V\setminus \Delta$ has length equal to $\delta|\Delta|$.  

\begin{lemma}\label{sullivanbounds}
 There exists $\tau=\tau(f)>0$ such that, for each $n\geq 1$, the $\tau$-scaled neighborhoods 
$V_{j,n}\supset  \Delta_{j,n}$, for $0\leq j\leq q_n-1$, are pairwise disjoint. 
\end{lemma}

This fact combined with 
Koebe distortion principle implies that $f^{k-j}:\Delta_{j,n}\to\Delta_{k,n}$ is a diffeomorphism with 
\emph{bounded distortion} for all $1 \leq j \leq k \leq q_n$. This is the contents of the following lemma, 
whose proof can be found in \cite[Section VI.2, Theorem 2.1, item 1]{demelovanstrien}.

\begin{lemma}\label{BoDist} There exists $K_1=K_1(f)>1$ such that:
\begin{equation}\label{boudist}
\frac{1}{K_1}\frac{|\Delta_{k,n}|}{|\Delta_{j,n}|}\leq\big|Df^{k-j}(x)\big|\leq K_1\frac{|\Delta_{k,n}|}{|\Delta_{j,n}|}
\end{equation}
for all $x\in\Delta_{j,n}$ and all $n \geq 0$.
\end{lemma}

The proposition below follows from the fact that the successive renormalizations of $f$ form a bounded sequence 
in the $C^1$-topology, which in turn is a consequence of Lemma \ref{sullivanbounds} and Lemma \ref{BoDist}. 
Nevertheless, we provide a proof as a courtesy to the reader. 

\begin{prop}\label{upperbound} There exists a constant $C_0=C_0(f)>1$ such that for all $n \geq 0$ and 
each $x\in\bigcup_{j=0}^{q_n-1}\Delta_{j,n}$ we have $\big|Df^{q_n}(x)\big| \leq C_0$. In particular $f$ does 
not satisfy the Collet-Eckmann condition.
\end{prop}

\begin{proof}[Proof of Proposition \ref{upperbound}] Since $0 \in I_0$ is a non-flat critical point with exponent 
$d>1$ there exist $0<a<b$ such that $a|x|^{d-1}\leq\big|Df(x)\big|\leq b|x|^{d-1}$ for all $x\in\Delta_{0,n}$ and 
all $n\in\nt$. In particular
\begin{equation}\label{naofl}
\big|Df(x)\big| \leq b|\Delta_{0,n}|^{d-1}
\end{equation}
for all $x\in\Delta_{0,n}$ and all $n\in\nt$.

Now, given $n\in\nt$, $j\in\{0,1,...,q_n-1\}$ and $x\in\Delta_{j,n}$ note that $f^{q_n-j}(x)\in\Delta_{0,n}$ and 
then $f^{q_n-j+1}(x)\in\Delta_{1,n}$. From \eqref{boudist} we deduce that:
\begin{equation}\label{a3}
\frac{1}{K_1}\frac{|\Delta_{0,n}|}{|\Delta_{j,n}|}\leq\big|Df^{q_n-j}(x)\big|\leq K_1\frac{|\Delta_{0,n}|}{|\Delta_{j,n}|}
\end{equation}
and also that
\begin{equation}\label{a4}
\frac{1}{K_1}\frac{|\Delta_{j,n}|}{|\Delta_{1,n}|}\leq\big|Df^{j-1}\big(f^{q_n-j+1}(x)\big)\big|\leq K_1\frac{|\Delta_{j,n}|}{|\Delta_{1,n}|}\,,
\end{equation}
while from \eqref{naofl} we have that
\begin{equation}\label{a5}
\big|Df\big(f^{q_n-j}(x)\big)\big| \leq b|\Delta_{0,n}|^{d-1}.
\end{equation}

From the chain rule and \eqref{a3}, \eqref{a4} and \eqref{a5} we get
\begin{equation}\label{a6}
\big|Df^{q_n}(x)\big| \leq b\,K_1^2\,\frac{|\Delta_{0,n}|^d}{|\Delta_{1,n}|}\,.
\end{equation}

Since $\big|Df(y)\big| \geq a|y|^{d-1}$ for all $y\in\Delta_{0,n}$ and since $d>1$ we get:
$$|\Delta_{1,n}|=\int_{0}^{|\lambda_n|}\!\big|Df(y)\big|\,dy \geq a\int_{0}^{|\lambda_n|}\!|y|^{d-1}dy=a\int_{0}^{|\lambda_n|}\!y^{d-1}dy=\frac{a}{d}\,|\lambda_n|^{d},$$that is, $|\Delta_{1,n}|\geq\frac{a}{d}\frac{1}{2^d}|\Delta_{0,n}|^d$. With this at hand and \eqref{a6} we finally obtain\footnote{Another way to obtain a positive lower bound for the sequence of ratios $|\Delta_{1,n}|/|\Delta_{0,n}|^d$ is by noting that $|\Delta_{1,n}|=\big|\phi(\Delta_{0,n})\big|^d=\big|D\phi(x_n)\big|^d|\Delta_{0,n}|^d$, where the $C^3$ local diffeomorphism $\phi$, which fixes the origin, is given by the definition of non-flat critical point (see the introduction) and $x_n$ is some point in $\Delta_{0,n}$.}$$\big|Df^{q_n}(x)\big| \leq \frac{b\,K_1^2\,d\,2^d}{a}=:C_0\quad\mbox{for all $x\in\Delta_{j,n}$.}$$
\end{proof}

Proposition \ref{upperbound} gives us large iterates where $|Df|$ is far away from infinity. With this at hand, 
we just need to find large iterates where $|Df|$ is far away from zero. For this purpose we give the following definition:

\begin{definition}\label{entrtimes} For each $x\in\mathcal{I}_f$ and each $n \geq 0$ we define the \emph{$n$-th entrance time} 
of $x$, denoted by $v_n(x)$, to be:$$v_n(x)=\min\big\{j \geq 0:f^j(x)\in\Delta_{0,n}\big\}.$$
\end{definition}

Note that $v_n(x)\in\{0,1,...,q_n-1\}$ and that $v_{n+1}(x) \geq v_n(x)$ for all $n \geq 0$ and all $x\in\mathcal{I}_f$. 
Note also that $v_n\big(f(c)\big)=q_n-1$ for all $n \geq 0$.

\begin{lemma}\label{novoA} For $\mu$-almost every $x\in\mathcal{I}_f$ we have $v_n(x)\to+\infty$ as $n\to+\infty$.
\end{lemma}

\begin{proof}[Proof of Lemma \ref{novoA}] For each $n \geq 0$ we define:$$A_n=\big\{x\in\mathcal{I}_f:v_{n+1}(x)>v_n(x)\big\}.$$

We also define:$$A=\limsup_{n\in\nt}A_n=\bigcap_{k \geq 0}\bigcup_{n=k}^{+\infty}A_n\,.$$

The set $A$ is $f$-invariant, and $x \in A$ if and only if $v_n(x)\to+\infty$ as $n\to+\infty$. We claim that $\mu(A_n)=1-1/a_n \geq 1/2$ for all $n \geq 0$.

Indeed, fix $n\in\nt$ and $j\in\{0,1,...,q_n-1\}$, and note that $v_n(x)=q_n-j$ for all $x\in\Delta_{j,n}\cap\mathcal{I}_f$. The intervals of the $(n+1)$-th level of renormalization contained in $\Delta_{j,n}$ are easily seen to be $\Delta_{j+kq_n,n+1}$ where $k\in\{0,1,...,a_n-1\}$.

Note that if $x\in\Delta_{j+kq_n,n+1}$ then $v_{n+1}(x)=q_{n+1}-(j+kq_n)=a_nq_n-(q_n-v_n(x)+kq_n)=v_n(x)+(a_n-1-k)q_n$. Thus, if $k=a_n-1$ then $v_{n+1}(x)=v_n(x)$, but if $k\in\{0,1,...,a_n-2\}$ then $v_{n+1}(x) \geq v_n(x)+q_n > v_n(x)$. Hence we have $A_n \cap \Delta_{j,n}=\bigcup_{k=0}^{a_n-2}\Delta_{j+kq_n,n+1}$ and then:$$\mu(A_n \cap \Delta_{j,n})=\sum_{k=0}^{a_n-2}\mu(\Delta_{j+kq_n,n+1})=(a_n-1)\mu(\Delta_{0,n+1})=\frac{a_n-1}{q_{n+1}}=\frac{1}{q_n}\left(1-\frac{1}{a_n}\right).$$

Thus $\mu(A_n)=\sum_{j=0}^{q_n-1}\mu(A_n \cap \Delta_{j,n})=1-\frac{1}{a_n}\geq\frac{1}{2}$ as was claimed. The claim implies that $\mu(A)\geq\frac{1}{2}>0$, and since $f$ is ergodic with respect to $\mu$ and the set $A$ is $f$-invariant, it follows that $\mu(A)=1$.
\end{proof}

We shall use the fact, due to Guckenheimer in the late seventies \cite{guck}, that at each renormalization level the interval containing the critical point is the largest (up to multiplication by a constant). More precisely:

\begin{lemma}\label{centralmaior} There exists a constant $\varepsilon=\varepsilon(f)>0$ such that $|\Delta_{0,n}|\geq\varepsilon|\Delta_{j,n}|$ for all $j\in\{0,1,...,q_n-1\}$ and all $n \geq 0$.
\end{lemma}

\begin{proof}[Proof of Lemma \ref{centralmaior}] This is proven in \cite[page 760]{deFariadeMeloPinto}; see also \cite[page 144]{guck}.
\end{proof}

The argument used by Guckenheimer relies on the Minimum Principle for maps with negative Schwarzian derivative (see \cite[Section II.6, Lemma 6.1]{demelovanstrien} for its statement) so at this point we need the fact that $f$ is of class $C^3$.

With Proposition \ref{upperbound}, Lemma \ref{novoA} and Lemma \ref{centralmaior} at hand we are ready to prove Theorem \ref{unbounded}:

\begin{proof}[Proof of Theorem \ref{unbounded}] The integrability of $\log|Df|$ was obtained by Przytycki \cite[Theorem B]{p93} so let us prove that $\int_{I_0}\!\log|Df|\,d\mu$ is zero. For this purpose define $A\subset\mathcal{I}_f$ to be the set of points $x\in\mathcal{I}_f$ such that $v_n(x)\to+\infty$ as $n\to+\infty$, where $\big\{v_n(x)\big\}$ is the sequence of entrance times of $x\in\mathcal{I}_f$, see Definition \ref{entrtimes}. We also define $B\subset\mathcal{I}_f$ to be the set of points $x\in\mathcal{I}_f$ such that:$$\lim_{n \to +\infty}\left\{\frac{\log\big|Df^n(x)\big|}{n}\right\}=\int_{I_0}\!\log|Df|\,d\mu\,,$$and note that $\mu(A \cap B)=1$ by Lemma \ref{novoA} and Birkhoff's ergodic theorem. By Proposition \ref{upperbound} we get$$\limsup_{n\in\nt}\left\{\frac{\log\big|Df^{q_n}(x)\big|}{q_n}\right\} \leq 0\quad\mbox{for all $x\in\mathcal{I}_f$,}$$and recall that from this it follows at once that $f$ does not satisfy the Collet-Eckmann condition. Moreover we already conclude that:
\begin{equation}\label{menor}
\lim_{n \to +\infty}\left\{\frac{\log\big|Df^{n}(x)\big|}{n}\right\} \leq 0\quad\mbox{for all $x \in B$.}
\end{equation}

Now let $x \in A$. For each $n\in\nt$, $f^{v_n(x)}$ maps $\Delta_{q_n-v_n(x),n} \ni x$ into $\Delta_{0,n}$. By Lemma \ref{BoDist} there exists $K_1=K_1(f)>1$ such that:$$\frac{1}{K_1}\frac{|\Delta_{0,n}|}{|\Delta_{q_n-v_n(x),n}|}\leq\big|Df^{v_n(x)}(x)\big|\leq K_1\frac{|\Delta_{0,n}|}{|\Delta_{q_n-v_n(x),n}|}\quad\mbox{for all $n \geq 1$.}$$

Applying Lemma \ref{centralmaior} we deduce in particular that for all $x \in A$ and for all $n \geq 1$:$$\big|Df^{v_n(x)}(x)\big|\geq\frac{\varepsilon}{K_1}>0$$and since $v_n(x)\to+\infty$ for all $x \in A$ it follows that:$$\liminf_{n\in\nt}\left\{\frac{\log\big|Df^{v_n(x)}(x)\big|}{v_n(x)}\right\} \geq 0\quad\mbox{for all $x \in A$.}$$

Therefore:
\begin{equation}\label{maior}
\lim_{n \to +\infty}\left\{\frac{\log\big|Df^{n}(x)\big|}{n}\right\} \geq 0\quad\mbox{for all $x \in A \cap B$.}
\end{equation}

Combining \eqref{menor} with \eqref{maior} we obtain$$\lim_{n \to +\infty}\left\{\frac{\log\big|Df^{n}(x)\big|}{n}\right\}=0\quad\mbox{for all $x \in A \cap B$}$$and since $\mu(A \cap B)=1$ this implies $\int_{I_0}\!\log|Df|\,d\mu=0$, as we wanted to prove.
\end{proof}

Let us finish Section \ref{sec:unimodal} with the following remark about the $\mu$-integrability of $\log|Df|$.

\begin{lemma}\label{remarkint} There exists $C=C(f)>1$ such that:
\begin{equation}\label{condint}
\frac{1}{C}\sum_{n\in\nt}\frac{1}{q_{n+1}}\log\frac{|\Delta_{0,n}|}{|\Delta_{0,n+1}|}\leq\int_{I_0}\!\big|\log|Df|\big|\,d\mu\leq C\sum_{n\in\nt}\frac{1}{q_{n}}\log\frac{|\Delta_{0,n}|}{|\Delta_{0,n+1}|}
\end{equation}
\end{lemma}

\begin{proof}[Proof of Lemma \ref{remarkint}] We just need to imitate the procedure of Section \ref{integrability}, during the proof of Proposition \ref{int}. Indeed, let $\psi=\big|\log|Df|\big|$ and for each $n\in\nt$ write $\psi_n=\psi\cdot\chi_{I_0\setminus\Delta_{0,n}}$. Using the power-law at the critical point $c=0$ we see that:
\begin{equation}\label{Intprimest}
\frac{1}{C_0}\log\frac{1}{|x|}\leq\psi(x)\leq C_0\log\frac{1}{|x|}
\end{equation}
for all $x\in\Delta_{0,n}\!\setminus\!\{0\}$ and $n\in\nt$ big enough. Now observe that:
\begin{equation}\label{Intsegest}
\int_{I_0}\!\psi_n\,d\mu=\int_{I_0\setminus\Delta_{0,n}}\!\psi_n\,d\mu=\sum_{j=0}^{n-1}\int_{\Delta_{0,j}\setminus\Delta_{0,j+1}}\!\psi_n\,d\mu\,.
\end{equation}

Since $|\Delta_{0,n+1}|\leq|x|\leq|\Delta_{0,n}|$ for all $x\in\Delta_{0,n}\!\setminus\!\Delta_{0,n+1}$ we see from \eqref{Intprimest} and \eqref{Intsegest} that:
\begin{equation}\label{Intterest}
\frac{1}{C_1}\sum_{j=0}^{n-1}\mu(\Delta_{0,j}\setminus\Delta_{0,j+1})\log\frac{1}{|\Delta_{0,j}|}\leq\int_{I_0}\!\psi_n\,d\mu\leq C_1\sum_{j=0}^{n-1}\mu(\Delta_{0,j}\setminus\Delta_{0,j+1})\log\frac{1}{|\Delta_{0,j+1}|}.
\end{equation}

Since $\displaystyle\mu(\Delta_{0,j}\!\setminus\!\Delta_{0,j+1})=\frac{1}{q_j}-\frac{1}{q_{j+1}}$ for all $j\in\nt$ we obtain from \eqref{Intterest} and a simple telescoping trick that:$$\frac{1}{C_2}\sum_{j=0}^{n-1}\frac{1}{q_{j+1}}\log\frac{|\Delta_{0,j}|}{|\Delta_{0,j+1}|}\leq\int_{I_0}\!\psi_n\,d\mu\leq C_2\sum_{j=0}^{n-1}\frac{1}{q_{j}}\log\frac{|\Delta_{0,j}|}{|\Delta_{0,j+1}|}$$and since $\psi_n\nearrow\psi$ as $n$ goes to infinity we are done.
\end{proof}

Note in particular that if $f$ is infinitely renormalizable of bounded type, then the series on the right-hand side of \eqref{condint} converges, since by the real bounds we know that $|\Delta_{0,n}|\asymp|\Delta_{0,n+1}|$ (the so-called \emph{bounded geometry}, see \cite[Section VI.2, Theorem 2.1, item 2]{demelovanstrien}), and the $q_n$'s grow exponentially fast. Note also that the integrability of $\log|Df|$ proved by Przytycki in \cite[Theorem B]{p93} already mentioned shows that, in the general case, the series on the left-hand side of \eqref{condint} always converges.

\section{Neutral measures on Julia sets}\label{ratapp} In this section we show some applications of Theorems \ref{main} and \ref{unbounded} to holomorphic dynamics. Let $f:\widehat{\C}\to\widehat{\C}$ be a rational map of the Riemann sphere with degree greater than or equal to two, and let $\mathcal{M}_f$ 
be the set of all $f$-invariant ergodic Borel probability measures in the Riemann sphere, whose support is contained in the Julia set of 
$f$. By \cite[Theorem A]{p93} we know that $\log|Df|$ is $\mu$-integrable for any $\mu\in\mathcal{M}_f$ (here $Df$ denotes the derivative 
considered with respect to the spherical metric). The \emph{Lyapunov exponent} of $f$ with respect to $\mu$ is the real 
number $\chi(\mu)=\int_{\widehat{\C}}\log{|Df|}\,d\mu$. By Birkhoff's Ergodic Theorem, 
\[
\chi(\mu)=\lim_{n \to +\infty}\left\{\frac{\log|Df^n(z)|}{n}\right\}
\] 
for $\mu$-almost every $z\in\widehat{\C}$. If $f$ is \emph{hyperbolic}, {\it i.e.} if each critical point is either periodic or 
contained in the basin of an attracting periodic orbit, then there exists $\chi>0$ such that $\chi(\mu)>\chi$ for any $\mu\in\mathcal{M}_f$. 
If $f$ is not hyperbolic, then either $f$ has a parabolic periodic orbit or one of its critical points lies in its Julia set (both phenomena can occur 
simultaneously). In the first case, the average of the Dirac measures supported along some parabolic periodic orbit gives an element 
$\mu\in\mathcal{M}_f$ such that $\chi(\mu)=0$ (take for example $f(z)=z^2+1/4$ and $\mu=\delta_{1/2}\in\mathcal{M}_f$). Being purely atomic, 
these measures are not so interesting. In the second case (when the Julia set contains a critical point) one may have a non-atomic measure $\mu\in\mathcal{M}_f$ such that $\chi(\mu)=0$. Following \cite{p93} we call such a measure \emph{neutral}.

As an example, consider the one-parameter family $f_{\omega}:\widehat{\C}\to\widehat{\C}$ of 
Blaschke products in the Riemann sphere given by:
\begin{equation}\label{defdelBlas}
f_{\omega}(z)=e^{2\pi i\omega}z^2\left(\frac{z-3}{1-3z}\right)\quad\mbox{for $\omega\in[0,1)$.}
\end{equation}

Just as any Blaschke product, every map in this family commutes with the geometric involution around the unit circle $\Phi(z)=1/\bar{z}$ 
(note that $\Phi$ is the identity in the unit circle). In particular every map in this family leaves invariant the unit circle (Blaschke products \emph{are} the rational maps leaving invariant the unit circle), and its restriction to $S^1$ is a real-analytic critical circle map with a unique 
critical point at $1$, which is of cubic type, and with critical value $e^{2\pi i\omega}$ (the fact that $f_{\omega}$ has topological degree one, when restricted to the unit circle, follows from the Argument Principle, since it has two zeros and one pole in the unit disk). By monotonicity of the rotation number (see for instance \cite{hermanihes} and \cite{demelovanstrien}) we know that for each irrational number $\theta$ in $[0,1)$ there exists a unique $\omega$ in $[0,1)$ such that the rotation number of $f_{\omega}|_{S^1}$ is $\theta$. 
As an application of Theorem \ref{main} we have the following result.

\begin{maintheorem}\label{exrat} Let $\omega$ in $[0,1)$ such that the rotation number of $f_{\omega}|_{S^1}$ is irrational. Then the rational 
function $f_{\omega}$ admits a non-atomic invariant ergodic Borel probability measure $\mu$ whose support is contained in the Julia set 
$J_{\omega}$ of $f_{\omega}$ and such that $\chi(\mu)=0$. Moreover, $f_{\omega}$ does not satisfy the Collet-Eckmann condition.
\end{maintheorem}

\begin{proof}[Proof of Theorem \ref{exrat}] The measure $\mu$ is precisely the unique invariant measure supported on the unit circle. By Theorem \ref{main}, we only need to prove that the unit circle is contained in the Julia set. Since $S^1$ is 
$f_{\omega}$-invariant and $f_{\omega}|_{S^1}$ is minimal (by Yoccoz's result \cite{yoccoz}), either we have $S^1 \subset J_{\omega}$, 
either $S^1 \cap J_{\omega}=\O$. Suppose, by contradiction, that $S^1 \cap J_{\omega}=\O$, and let $U$ be the Fatou component 
of $f_{\omega}$ containing $S^1$. By the invariance of the unit circle, $U$ is mapped into itself by $f_{\omega}$, and therefore it must be 
a Siegel disk or a Herman ring (precisely because it has an invariant simple closed curve on its interior). But in that case 
$f_{\omega}:U \to U$ would be a biholomorphism, which is impossible since it has a critical point in the unit circle. 
Therefore $S^1 \subset J_{\omega}$, and Theorem \ref{exrat} follows directly from Theorem \ref{main}.
\end{proof}

By the main result of \cite{PJuanSm}, and since there is only one critical point in the Julia set $J_{\omega}$, we deduce that 
$f_{\omega}$ does not satisfy any of the standard definitions of non-uniform hyperbolicity for rational maps in the Riemann 
sphere (topological Collet-Eckmann condition, uniform hyperbolicity of periodic orbits in the Julia set, etc.).

We remark that Theorem \ref{unbounded} in \S \ref{sec:unimodal} yields analogous examples to those given in 
Theorem \ref{exrat} above, in the context of polynomials. For instance, one can take $f(z)=z^2+c_{\infty}$, 
where $c_{\infty}\in[-2,1/4]$ denotes the infinitely renormalizable parameter for period doubling 
(the so-called \emph{Feigenbaum} parameter). In that case, the neutral measure is the unique invariant measure 
supported in the closure of the post-critical set of $f$, which is non-atomic as explained in Section \ref{sec:unimodal}.

\subsection{Other examples} Let $\theta\in[0,1]$ be an irrational number, and consider the quadratic polynomial 
$P_{\theta}:\C\to\C$ given by:
\begin{equation}\label{peteta}
P_{\theta}(z)=e^{2\pi i\theta}z+z^2\,.
\end{equation}

The origin is a fixed point of $P_{\theta}$, and $DP_{\theta}(0)=e^{2\pi i\theta}$ has modulus one. Let us assume 
that $\theta$ is of \emph{bounded type}, that is, there exists $\varepsilon>0$ such that:
$$\left|\theta-\frac{p}{q}\right|\geq\frac{\varepsilon}{q^2}\,,$$for any integers $p$ and $q \neq 0$. 
By a classical result of Siegel $P_{\theta}$ is \emph{linearizable} around the origin: there exists a simply-connected 
component $\Omega_{\theta}$ of the Fatou set of $P_{\theta}$, a \emph{Siegel disk}, that contains the origin and where 
$P_{\theta}$ is conformally conjugate to its linear part $z \mapsto e^{2\pi i\theta}z$, an irrational rotation acting 
on the unit disk. A famous theorem of Douady \cite{douady} (see also the recent paper \cite{zhang}) asserts that 
$\partial\,\Omega_{\theta}$ is a quasi-circle that contains the (unique) critical point of $P_{\theta}$. 
Moreover, there exist $\omega=\omega(\theta)\in(0,1)$ and a quasiconformal homeomorphism 
$\phi_{\theta}:\C\to\C$ with $\phi_{\theta}(\D)=\Omega_{\theta}$ and such that 
$\phi_{\theta} \circ f_{\omega}=P_{\theta}\circ\phi_{\theta}$ on $\C\setminus\D$, where $f_{\omega}$ 
is the Blaschke product given by \eqref{defdelBlas}. In particular $\phi_{\theta}:S^1\to\partial\,\Omega_{\theta}$ 
is a quasisymmetric homeomorphism which conjugates $f_{\omega}$ with $P_{\theta}$. Combining this with 
Theorem \ref{exrat} we get the following result.

\begin{coro}\label{corosiegel} Let $\theta\in[0,1]$ be an irrational number of bounded type, and let 
$P_{\theta}$ be the quadratic polynomial given by \eqref{peteta}. Then the harmonic measure on the boundary 
of the Siegel disk $\Omega_{\theta}$, viewed from the origin, is a neutral measure for $P_{\theta}$.
\end{coro}

Corollary \ref{corosiegel} follows from Theorem \ref{exrat} and the general fact that a quasiconformal conjugacy 
between rational 
maps preserves zero Lyapunov exponents (one way to see this is to combine Koebe's distortion lemma with the fact 
that a quasiconformal homeomorphism is bi-H\"older and the argument in \cite[Lemma 8.3]{rlp}).

More generally, the harmonic measure at the boundary of any Siegel disk or at the boundary of any Herman ring is a 
neutral measure, as stated in the introduction of \cite{p93}. In particular, by taking some polynomial or 
rational function $f$ with a Siegel disk whose boundary does not contain critical points (for such an example see 
\cite{abc} and the references therein) we obtain $\mu\in\mathcal{M}_f$ such that $\chi(\mu)=0$, $\mu$ has no atoms 
and $\supp(\mu)$ contains no critical point of $f$. Note, however, that for any rational function $f$ and any neutral 
measure $\mu$ we have that $\supp(\mu)$ must be contained in the closure of the forward orbit of some critical point 
of $f$, otherwise $f$ would be expanding on $\supp(\mu)$ and then $\mu$ would not be neutral.

Finally, let us mention that Cortez and Rivera-Letelier have constructed in \cite{juan1} and \cite{juan2} 
real quadratic polynomials for which the $\omega$-limit set of the critical point is a minimal Cantor 
set (necessarily contained in the Julia set) supporting \emph{any prescribed} number of neutral measures 
(finite, countable infinite or uncountable). This is in sharp contrast with the examples discussed above, where $f|_{\supp(\mu)}$  is uniquely ergodic.

\section{Further questions}\label{sec:final}

We conclude with some questions (in addition to Question \ref{expcrit} posed at the end of Section \ref{sec:proofofmain}).

Let $f$ be a rational map of the Riemann sphere with degree greater than or equal to two, and let $\mu$ be a 
neutral measure for $f$, that is, a non-atomic $f$-invariant ergodic Borel probability measure, whose support is 
contained in the Julia set of $f$, and with Lyapunov exponent equal to zero (just as in the situation of Theorem \ref{exrat}).

\begin{question} Is it true that $h_{\htop}(f|_{\supp(\mu)})=0$? Is $f|_{\supp(\mu)}$ a minimal dynamical system? 
Is it true, at least, that $\supp(\mu)$ has no periodic orbits?
\end{question}


\begin{question} What are all the examples of rational maps having some non-atomic invariant ergodic Borel probability 
measure, supported inside its Julia set, and with zero Lyapunov exponent?
\end{question}

A difficult problem in the context of critical circle maps with finitely many critical points is the one of 
\emph{geometric rigidity}. More precisely, let $f$ and $g$ be two orientation preserving $C^3$ circle homeomorphisms 
with the same irrational rotation number, and with $N \geq 1$ non-flat critical points of odd type. 
Denote by $S_f=\{c_1,...,c_N\}$ the ordered critical set of $f$, by $S_g=\{c'_1,...,c'_N\}$ the ordered 
critical set of $g$, and suppose that the criticalities of $c_i$ and $c'_i$ are the same for all 
$i \in \{1,...,N\}$ (the cubic case is the generic one). Finally, denote by $\mu_f$ and $\mu_g$ the corresponding 
unique invariant measures of $f$ and $g$.

By Yoccoz's result \cite{yoccoz} we know that $f$ and $g$ are topologically conjugate to each other. 
By elementary reasons, the condition $\mu_f\big([c_i,c_{i+1}]\big)=\mu_g\big([c'_i,c'_{i+1}]\big)$ 
for all $i \in \{1,...,N-1\}$ is necessary (and sufficient) in order to have a topological conjugacy 
between $f$ and $g$ that sends the critical points of $f$ to the critical points of $g$. Under this assumption, 
it turns out that this conjugacy is in fact a \emph{quasisymmetric} homeomorphism. This follows from a recent 
general result of Clark and van Strien \cite{treseb}. In this context it also follows from the real bounds 
(see \cite[Corollary 4.6]{edsonwelington1} for the case of a single critical point, and the forthcoming 
article \cite{estevezdefaria} for the multicritical case).

\begin{question}\label{QS} Is this conjugacy a smooth diffeomorphism?
\end{question}

In the case of exactly one critical point this question has been answered in the affirmative, in the real-analytic category.
This is due to the efforts of several mathematicians during the last twenty years (see \cite{avila}, 
\cite{tesisedson}, \cite{edson}, \cite{edsonwelington1}, \cite{edsonwelington2}, \cite{khaninteplinsky}, \cite{khmelevyampolsky}, \cite{yampolsky1}, \cite{yampolsky2}, \cite{yampolsky3} and \cite{yampolsky4}). A positive answer was recently obtained also in the $C^3$ category (see \cite{tesegua}, \cite{GMdM} and \cite{guamelo}). To the best of our knowledge, the case of more than one critical point remains completely open.

\section*{Acknowledgements}

We wish to thank Juan Rivera-Letelier for very useful comments about Section \ref{ratapp}, especially for pointing to 
us Corollary \ref{corosiegel}. We are also grateful to Trevor Clark, Gabriela Estevez, Katrin Gelfert, 
Mario Ponce and Charles Tresser for several conversations about these and related matters. Finally, we would like 
to thank the anonymous referee for his/her keen remarks and questions, which have lead to a considerable 
improvement of our paper.


\begin{thebibliography}{[ABD]}
\bibitem{avila} Avila, A., On rigidity of critical circle maps, {\em Bull. Braz. Math. Soc.}, {\bf 44}, 611-619, (2013).
\bibitem{abc} Avila, A., Buff, X., Ch\'eritat, A., Siegel disks with smooth boundaries, {\em Acta Math.}, {\bf 193}, 1-30, (2004).
\bibitem{acip} Bruin, H., Rivera-Letelier, J., Shen, W., van Strien, S., Large derivatives, backward contraction and invariant densities for interval maps, {\em Invent. Math.}, {\bf 172}, 509-533, (2008).
\bibitem{treseb} Clark, T., van Strien, S., Quasisymmetric rigidity in one-dimensional dynamics, {\em manuscript}.
\bibitem{juan1} Cortez, M., Rivera-Letelier, J., Invariant measures of minimal post-critical sets of logistic maps, {\em Israel. J. Math.}, {\bf 176}, 157-193, (2010).
\bibitem{juan2} Cortez, M., Rivera-Letelier, J., Choquet simplices as spaces of invariant probability measures on post-critical sets, {\em Ann. Henri Poincar\'e}, {\bf 27}, 95-115, (2010).
\bibitem{douady} Douady, A., Disques de Siegel et aneaux de Herman, {\em S\'em. Bourbaki 1986/87}, {\em Ast\'erisque}, {\bf 152-153}, 151-172, (1987).
\bibitem{estevezdefaria} Estevez, G., de Faria, E., Real bounds and quasisymmetric rigidity for multicritical circle maps, {\em in preparation}.
\bibitem{tesisedson} de Faria, E., Proof of universality for critical circle mappings, {\em Ph.D. Thesis}, CUNY, 1992.
\bibitem{edson} de Faria, E., Asymptotic rigidity of scaling ratios for critical circle mappings, {\em Ergod. Th. \& Dynam. Sys.}, {\bf 19}, 995-1035, (1999).
\bibitem{edsonwelington1} de Faria, E., de Melo, W., Rigidity of critical circle mappings I, {\em J. Eur. Math. Soc.}, {\bf 1}, 339-392, (1999).
\bibitem{edsonwelington2} de Faria, E., de Melo, W., Rigidity of critical circle mappings II, {\em J. Amer. Math. Soc.}, {\bf 13}, 343-370, (2000).
\bibitem{deFariadeMeloPinto} de Faria, E., de Melo, W., Pinto, A., Global hyperbolicity of renormalization for $C^r$ 
unimodal mappings, {\em Ann. of Math.}, {\bf 164}, 731-824, (2006).  
\bibitem{tesegua} Guarino, P., Rigidity conjecture for $C^3$ critical circle maps, {\em Ph.D. Thesis}, IMPA, 2012, available at {\tt{www.preprint.impa.br\/}}.
\bibitem{GMdM} Guarino, P., Martens, M., de Melo, W., Rigidity of smooth critical circle maps: unbounded combinatorics, {\em in preparation}.
\bibitem{guamelo} Guarino, P., de Melo, W., Rigidity of smooth critical circle maps, available at {\tt{arXiv:1303.3470\/}}.
\bibitem{guck} Guckenheimer, J., Sensitive Dependence to Initial Conditions for One Dimensional Maps, {\em Commun. Math. Phys.}, {\bf 70}, 133-160, (1979).
\bibitem{hall} Hall, G. R., A $C^\infty$ Denjoy counterexample, {\em Ergod. Th. \& Dynam. Sys.}, {\bf 1}, 261-272, (1981).
\bibitem{hermanihes} Herman, M., Sur la conjugaison diff\'erentiable des diff\'eomorphismes du cercle \`a des rotations, {\em Publ. Math. IHES}, {\bf 49}, 5-233, (1979).
\bibitem{herman} Herman, M., Conjugaison quasi-sim\'etrique des hom\'eomorphismes du cercle \`a des rotations (manuscript), (1988).
\bibitem{katokhass} Katok, A., Hasselblatt, B., {\em Introduction to the Modern Theory of Dynamical Systems}, Cambridge University Press, 1995.
\bibitem{keller} Keller, G., Exponents, attractors and Hopf decompositions for interval maps, {\em Ergod. Th. \& Dynam. Sys.}, {\bf 10}, 717-744, (1990).
\bibitem{khaninteplinsky} Khanin, K., Teplinsky, A., Robust rigidity for circle diffeomorphisms with singularities, {\em Invent. Math.}, {\bf 169}, 193-218, (2007).
\bibitem{khin} Khinchin, A. Ya., {\em Continued fractions}, (reprint of the 1964 translation), Dover Publications, Inc., 1997.
\bibitem{khmelevyampolsky} Khmelev, D., Yampolsky, M., The rigidity problem for analytic critical circle maps, {\em Mosc. Math. J.}, {\bf 6}, 317-351, (2006).
\bibitem{lang} Lang, S., {\em Introduction to diophantine approximations}, (new expanded edition), Springer-Verlag, (1995).
\bibitem{lishen} Li, S., Shen, W., Hausdorff dimension of Cantor attractors in one-dimensional dynamics, {\em Invent. Math.}, {\bf 171}, 345-387, (2008).
\bibitem{demelovanstrien} de Melo, W., van Strien, S., {\em One-dimensional Dynamics}, Springer-Verlag, New York, (1993).
\bibitem{nowsands} Nowicki, T., Sands, D., Non-uniform hyperbolicity and universal bounds for S-unimodal maps, {\em Invent. Math.}, {\bf 132}, 633-680, (1998).
\bibitem{petersen} Petersen, C. L., The Herman-\'Swi\c{a}tek theorems with applications, {\em The Mandelbrot set, theme and variations}, 211-225, London Math. Soc. Lecture Note Ser., {\bf 274}, Cambridge Univ. Press, Cambridge, (2000).
\bibitem{p93} Przytycki, F., Lyapunov characteristic exponents are nonnegative, {\em Proc. of the A.M.S.}, {\bf 119}, 309-317, (1993).
\bibitem{PJuanSm} Przytycki, F., Rivera-Letelier, J., Smirnov, S., Equivalence and topological invariance of conditions for non-uniform hyperbolicity in the iteration of rational maps, {\em Invent. Math.}, {\bf 151}, 29-63, (2003).
\bibitem{rlp} Przytycki, F., Rivera-Letelier, J., Statistical properties of topological Collet-Eckmann maps, 
{\em Ann. Scient. \'Ec. Norm. Sup.\/}, {\bf 40}, 135-178, (2007).
\bibitem{livroconf} Przytycki, F., Urba\'nski, M., {\em Conformal Fractals: Ergodic Theory Methods}, London Mathematical Society Lecture Note Series, {\bf 371}, Cambridge University Press, (2010).
\bibitem{juan} Rivera-Letelier, J., Asymptotic expansion of smooth interval maps, available at {\tt{arXiv:1204.3071\/}}.
\bibitem{sullivan} Sullivan, D., Bounds, quadratic differentials, and renormalization conjectures, 
{\em AMS Centennial Publications}, {\bf{2}}, Mathematics into the Twenty-first Century, (1988). 
\bibitem{swiatek} \'Swi\c{a}tek, G., Rational rotation numbers for maps of the circle, {\em Commun. Math. Phys.}, {\bf 119}, 109-128, (1988).
\bibitem{yampolsky1} Yampolsky, M., Complex bounds for renormalization of critical circle maps, {\em Ergod. Th. \& Dynam. Sys.}, {\bf 19}, 227-257, (1999).
\bibitem{yampolsky2} Yampolsky, M., The attractor of renormalization and rigidity of towers of critical circle maps, {\em Commun. Math. Phys.}, {\bf 218}, 537-568, (2001).
\bibitem{yampolsky3} Yampolsky, M., Hyperbolicity of renormalization of critical circle maps, {\em Publ. Math. IHES}, {\bf 96}, 1-41, (2002).
\bibitem{yampolsky4} Yampolsky, M., Renormalization horseshoe for critical circle maps, {\em Commun. Math. Phys.}, {\bf 240}, 75-96, (2003).
\bibitem{yoccoz} Yoccoz, J.-C., Il n'y a pas de contre-exemple de Denjoy analytique, {\em C.R. Acad. Sc. Paris}, {\bf 298}, 141-144, (1984).
\bibitem{zhang} Zhang, G., All bounded type Siegel disks of rational maps are quasi-disks, {\em Invent. Math.}, {\bf 185}, 421-466, (2011).
\end{thebibliography}
\end{document}